\documentclass{article}
\usepackage[numbers,sort]{natbib}

\usepackage{amsmath}
\usepackage{amssymb}
\usepackage{amsthm}
\usepackage{tikz}
\usetikzlibrary{shapes.geometric}
\usepackage{caption}
\usepackage[hidelinks]{hyperref}
\usepackage{authblk}
\usepackage[T1]{fontenc}

\DeclareMathOperator{\fs}{fs}
\DeclareMathOperator{\elm}{elm}

\newtheorem{theorem}{Theorem}

\newtheorem{proposition}[theorem]{Proposition}
\newtheorem{lemma}[theorem]{Lemma}
\newtheorem{corollary}[theorem]{Corollary}

\newcount\savedtheorem

\def\claimqed{\smash{\scalebox{.75}[0.75]{$(\square$)}}}

\author{Hany Ibrahim \qquad  Peter Tittmann\\University of Applied Sciences Mittweida}
\title{Edge Contraction and Line Graphs}
\date{}

\begin{document}

\maketitle

\begin{abstract}
	Given a family of graphs $\mathcal{H}$, a graph $G$ is $\mathcal{H}$-free if any subset of $V(G)$ does not induce a subgraph of $G$ that is isomorphic to any graph in $\mathcal{H}$. We present sufficient and necessary conditions for a graph $G$ such that $G/e$ is $\mathcal{H}$-free for any edge $e$ in $E(G)$. Thereafter, we use these conditions to characterize claw-free and line graphs.	
\end{abstract}

\section{Introduction}
%
A graph $G$ is an ordered pair $(V(G),E(G))$ where $V(G)$ is a set of vertices and $E(G)$ is a set of $2$-elements subsets of $V(G)$ called edges. The set of all graphs is $\mathcal{G}$. The degree of a vertex $v$, denoted by $deg(v)$, is the number of edges incident to $v$. We denote the maximum degree of a vertex in a graph $G$ by $\Delta(G)$. We call two vertices adjacent if there is an edge between them, otherwise, we call them nonadjacent. Moreover, the set of all vertices adjacent to a vertex $v$ is called the \emph{neighborhood} of $v$, which we denote by $N(v)$. On the other hand, the \emph{closed neighborhood} of $v$, denoted by $N[v]$, is $N(v) \cup \{v\}$. Generalizing this to a set of vertices $S$, the neighborhood of $S$, denoted by $N(S)$, is defined by $N(S) := \bigcup_{v\in S} N(v) - S$. Similarly the closed neighborhood of $S$, denoted by $N[S]$, is $N(S) \cup S$. Moreover, for a subset of vertices $S$, we denote the set of vertices in $S$ that are adjacent to $v$ by $N_{S}(v)$. Furthermore, we write $v$ is adjacent to $S$ to mean that $S \subseteq N(v)$ and $v$ is adjacent to exactly $S$ to mean that $S = N(v)$.

A set of vertices $S$ is \emph{independent} if there is no edge between any two vertices in $S$. We call a set $S$ \emph{dominating} if $N[S] = V(G)$. A subgraph $H$ of a graph $G$ is a graph where $V(H) \subseteq V(G)$ and $E(H) \subseteq E(G)$. An \emph{induced} graph $G[S]$ for a given set $S \subseteq V$, is a subgraph of $G$ with vertex set $S$ and two vertices in $G[S]$ are adjacent if and only if they are adjacent in $G$. Two graphs $G,H$ are \emph{isomorphic} if there is a bijective mapping $f: V(G) \to V(H)$ where $u,v \in V(G)$ are adjacent if and only if $f(u),f(v)$ are adjacent in $H$. In this case we call the mapping $f$ an isomorphism. Two graphs that are not isomorphic are called \emph{non-isomorphic}. In particular, an isomorphism from a graph to itself is called \emph{automorphism}. Furthermore, two vertices $u,v$ are similar in a graph $G$ if there is an automorphism that maps $u$ to $v$. The set of all automorphisms of a graph $G$ forms a group called the automorphism group of $G$, denoted by \emph{Aut($G$)}. The complement of a graph $G$, denoted by $\bar{G}$, is a graph with the same vertex set as $V(G)$ and two vertices in $\bar{G}$ are adjacent if and only if they are nonadjacent in $G$.

The \emph{independence number} of a graph $G$, denoted by $\alpha(G)$, is the largest cardinality of an independent set in $G$. In this paper, we write singletons $\{x\}$ just as $x$ whenever the meaning is clear from the context. A vertex $u$ is a \emph{corner dominated} by $v$ if $N[u] \subseteq N[v]$.
Let $\mathcal{H}$ be a set of graphs. A graph $G$ is called \emph{$\mathcal{H}$-free} if there is no induced subgraph of $G$ that is isomorphic to any graph in $\mathcal{H}$, otherwise, we say $G$ is \emph{$\mathcal{H}$-exist}. 

By \emph{contracting} the edge between $u$ and $v$, we mean the graph constructed from $G$ by adding a vertex $w$ with edges from $w$ to the union of the neighborhoods of $u$ and $v$, followed by removing $u$ and $v$. We denote the graph obtained from contracting $uv$ by $G/uv$. If $e$ is the edge between $u$ and $v$, then we also denote the graph $G/uv$ by $G/e$. Further, we call $G/e$ a $G$-contraction. 
Finally, any graph in this paper is simple. For notions not defined, please consult \cite{bondy2000graph}. We divide longer proofs into smaller claims, and we prove them only if their proofs are not apparent.

For a graph invariant $c$, a graph $G$, and a $G$-contraction $H$, the question of how $c(G)$ differs from $c(H)$ is investigated for different graph invariants. For instance, how contracting an edge in a graph affects its $k$-connectivity. Hence, the intensively investigated (\cite{kriesell2002survey}) notion of a \emph{$k$-contractible} edge in a $k$-connected graph $G$ is defined as the edge whose contraction yields a $k$-connected graph. Another example is in the game Cops and Robber where a policeman and a robber are placed on two vertices of a graph in which they take turns to move to a neighboring vertex. For any graph $G$, if the policeman can always end in the same vertex as the robber, we call $G$ \emph{cop-win}. However, $G$ is \emph{$CECC$} if it is not cop-win, but any $G$-contraction is cop-win. The characteristics of a $CECC$ graph are studied in \cite{cameron2015edge}.

A further example is the investigation of the so-called \emph{contraction-critical edges}, with respect to independence number, that is, an edge $e$ in a graph $G$ where $\alpha(G/e) \leq \alpha(G)$, studied in \cite{plummer2014note}. Furthermore, the case where $c$ is the chromatic and clique number, respectively, has been investigated in \cite{diner2018contraction, paulusma2016reducing, paulusma2019critical}.

In this article, we investigate the graph invariant $H$-free for a given set of graphs $\mathcal{H}$. In particular, we present sufficient and necessary conditions for a graph $G$ such that any $G$-contraction is $\mathcal{H}$-free.

Let $\mathcal{H}$ be a set of graphs. The set of \emph{elementary (minimal)} graphs in $\mathcal{H}$, denoted by \emph{$\elm$($\mathcal{H}$)}, is defined as $\{H \in \mathcal{H}: $ if $ G \in \mathcal{H} $ and $H$ is $G$-exist, then $G$ is isomorphic to $H\}$.
From the previous definition, we can directly obtain the following.
\begin{proposition}{\label{prop: elem}}
	Let $\mathcal{H}$ be a set of graphs. Graph $G$ is $\mathcal{H}$-free if and only if G is $\elm(\mathcal{H})$-free.
\end{proposition}

We call an $\mathcal{H}$-free graph $G$, \emph{strongly $\mathcal{H}$-free} if any $G$-contraction is $\mathcal{H}$-free. Furthermore, an $\mathcal{H}$-exist graph $G$ is a \emph{critically $\mathcal{H}$-exist} if any $G$-contraction is $\mathcal{H}$-free.
If we add any number of isolated vertices to a strongly $H$-free or critically $H$-exist graph, we obtain a graph with same property. Thus, from this section and forward, we exclude graphs having isolated vertices unless otherwise stated.

We conclude directly the following.
\begin{proposition}{\label{Theore: forbidden}}
	Let $\mathcal{H}$ be a set of graphs and $G$ be a graph where $G$ is neither critically $\mathcal{H}$-exist nor $\mathcal{H}$-free but not strongly $\mathcal{H}$-free. The graph $G$ is $\mathcal{H}$-free if and only if any $G$-contraction is $\mathcal{H}$-free.
\end{proposition}

Given a graph $G$ and a set of graphs $\mathcal{H}$, we call $G$ \emph{$\mathcal{H}$-split} if there is a $G$-contraction isomorphic to a graph in $\mathcal{H}$. Furthermore, $G$ is \emph{$\mathcal{H}$-free-split} if $G$ is $\mathcal{H}$-split and $\mathcal{H}$-free. Moreover, the set of all $\mathcal{H}$-free-split graphs, for a given $\mathcal{H}$, is denoted by \emph{$\fs(\mathcal{H})$}.

\begin{proposition}{\label{Theorem: key}}
	Let $\mathcal{H}$ be a set of graphs and $G$ be a $\mathcal{H}$-free graph. Then $G$ is strongly $\mathcal{H}$-free if and only if $G$ is $\fs(\mathcal{H})$-free.
\end{proposition}
\begin{proof}
	Assume for the sake of contradiction that there exists a strongly $\mathcal{H}$-free graph $G$ with an induced $\mathcal{H}$-free-split subgraph $J$. Consequently, there is an edge $e$ in $J$ such that $J/e$ induces a graph in $\mathcal{H}$. As a result, $G/e$ is $\mathcal{H}$-exist, which contradicts the fact that $G$ is strongly $\mathcal{H}$-free. 
	
	In contrast, if $G$ is an $\mathcal{H}$-free but not a strongly $\mathcal{H}$-free, then there is a set $U \subseteq V(G)$ such that there is an edge $e \in E(G[U])$ where $G/e$ is $\mathcal{H}$-exist. Let $U$ be a minimum set with such a property. Thus $G[U]$ is $\mathcal{H}$-free-split.
\end{proof}

From Propositions \ref{Theore: forbidden} and \ref{Theorem: key}, we deduce the following.
\begin{theorem}{\label{Theorem: characeterization}}
	Let $\mathcal{H}$ be a set of graphs and $G$ be a $\fs(\mathcal{H})$-free graph where $G$ is not critically $\mathcal{H}$-exist. The graph $G$ is $\mathcal{H}$-free if and only if any $G$-contraction is $\mathcal{H}$-free.
\end{theorem}

Theorem \ref{Theorem: characeterization} provides a sufficient and necessary condition that answers the question we investigate in this paper, however, it translates this problem to a new one, namely determining characterizations for critically $\mathcal{H}$-exist and $\mathcal{H}$-free-split graphs for a set of graphs $\mathcal{H}$. In Sections \ref{Section: CC(H)} and \ref{Section: EC(H)}, we present some properties for these families of graphs.

\subsection{The \texorpdfstring{$\mathcal{H}$}{H}-Split Graphs}{\label{Section: CC(H)}}
Let $H$ be a graph with $v \in V(H)$ and $N_{H}(v) = U \cup W$. The \emph{$splitting(H,v,U,W)$} is the graph obtained from $H$ by removing $v$ and adding two vertices $u$ and $w$ where $ N_{H}(u)= U \cup \{w\}$ and $ N_{H}(w)= W \cup \{u\}$. Furthermore, \emph{$splitting(H,v)$} is the set of all graphs for any possible $U$ and $W$. Moreover, \emph{$splitting(H)$} is the union of the $splitting(H,v)$ for any vertex $v \in V(H)$. Given a set of graphs $\mathcal{H}$, \emph{$splitting(\mathcal{H})$} is the union of the splittings of every graph in $\mathcal{H}$.

\begin{theorem}{\label{Theorem: CC=H[]}}
	For a graph $G$ and a set of graphs $\mathcal{H}$, $G$ is an $\mathcal{H}$-split if and only if $G \in splitting(\mathcal{H})$.
\end{theorem}
\begin{proof}
	Let $G$ be an $\mathcal{H}$-split. Hence there is a graph $H \in \mathcal{H}$ such that $G$ is $H$-split.
	Thus, there are two vertices $u,w \in V(G)$ such that $G/uw$ is isomorphic to $H$. Let $x := V(G/uw)-V(G)$, then $N_{G/uw}(x) = (N_{G}(u) \cup N_{G}(w)) - \{u,w\}$. As a result, $G \in splitting(H,x)$. Consequently, $G \in splitting(\mathcal{H})$.

	Conversely, let $G \in splitting(\mathcal{H})$. Hence there is a graph $H \in \mathcal{H}$ such that $G \in splitting(H)$.
	Thus, there are two adjacent vertices $u,w \in V(G)$ such that $G/uw \cong H$. Thus, $G$ is $\mathcal{H}$-split.
\end{proof}

For a set of graphs $\mathcal{H}$ and using Theorem \ref{Theorem: CC=H[]}, we can use $splitting(\mathcal{H})$ to construct all $\mathcal{H}$-split graphs, consequently $\mathcal{H}$-free-split graphs.

\begin{proposition}
	In a graph $G$, let $u,v \in V(G)$. If $u$ is similar to $v$, then $splitting(G,u) = splitting(G,v)$.
\end{proposition}
By the previous proposition, for a graph $H$, the steps to construct the $H$-free-split graphs are:
\begin{itemize}
	\item Let $\pi$ be the partition of $V(H)$ induced by the orbits generated from $Aut(H)$;
	\item for every orbit $o \in \pi$, we choose a vertex $v \in o$; and
	\item construct $splitting(H,v)$.
\end{itemize}

\begin{proposition}\label{pro: split graphs that are not free}
	Let $G$ be a graph, $v$ a vertex in $V(G)$ where $N_{G}(v) = U \cup W$. If $U=N_{G}(v)$ or $W=N_{G}(v)$, then $splitting(G,v,U,W)$ is not $G$-free-split.
\end{proposition}
\begin{proposition}\label{pro: split graphs that are not free with degree one}
	Let $G$ be a graph and $v$ a vertex in $V(G)$. If $deg(v)=1$, then $splitting(G,v)$ contains no $G$-free-split graph.
\end{proposition}

\begin{proposition}\label{pro: no free-split for paths}
	If $G$ is a path, then $splitting(G)$ contains no $G$-free-split graph.
\end{proposition}

\begin{proposition}\label{pro: only one free-split for cycles}
	If $G$ is a $C_{n}$ for an integer $n \geq 3$, then the $G$-free-split is $C_{n+1}$.
\end{proposition}

\subsection{Critically \texorpdfstring{$\mathcal{H}$}{H}-Exist Graphs}{\label{Section: EC(H)}}

\begin{theorem}{\label{Theorem: critical H-exist: not S is independent}}
	Let $G$ be a graph and $\mathcal{H}$ be a set of graphs.
	If $G$ is a critically $\mathcal{H}$-exist, then for any $S \subseteq V(G)$ such that $G[S]$ is isomorphic to a graph in $\mathcal{H}$, the followings properties hold:
	\begin{enumerate}
		\item $V(G) \setminus S$ is independent and
		\item there is no corner in $V(G) \setminus S$ that is dominated by a vertex in $S$.
	\end{enumerate}
\end{theorem}
\begin{proof}
	\begin{enumerate}
		\item For the sake of contradiction, assume there is a $S \subseteq V(G)$ such that $G[S]$ is isomorphic to a graph $H \in \mathcal{H}$  but $V(G) \setminus S$ is not independent. Hence, there are two vertices $u,v \in V(G) \setminus S$ where $u$ and $v$ are adjacent. Consequently, $G/uv[S]$ is isomorphic to $H$, which contradicts the fact that $G$ is a critically $\mathcal{H}$-exist.

		\item Since $V(G) \setminus S$ is independent, the neighborhood of any vertex in $V(G) \setminus S$ is a subset of $S$. For the sake of contradiction, assume that there is a corner $u \in V(G) \setminus S$ that is dominated by $v \in S$. However, $G/uv[S]$ is isomorphic to a graph $H \in \mathcal{H}$, which contradicts the fact that $G$ is a critically $\mathcal{H}$-exist.
	\end{enumerate}
\end{proof}

\begin{corollary}\label{coro: no vertex adajcet to 1 or 2 or 3 or max degree in critical graph}
	Let $G$ be a critically $\mathcal{H}$-exist graph for a set of graphs $\mathcal{H}$. If $S$ is a vertex set that induces a graph in $\mathcal{H}$, then no vertex in $V(G) \setminus S$ is adjacent to exactly one vertex, two adjacent vertices, three vertices that induce either $P_{3}$ or $C_{3}$, or a vertex with degree $|V(G)|-1$.
\end{corollary}

Let $G$ be a graph with adjacent vertices $u,v$, and $\{w\} := V(G/uv) - V(G)$.
We define the mapping $f:2^{V(G)} \to 2^{V(G/uv)}$ as follows:
\[
f(S) = 
\begin{cases}
	S & \text{if }u,v \notin S,\\
	(S \cup \{w\}) - \{u,v\} & \text{otherwise.} \\ 
\end{cases}
\]
Let $S$ be a vertex set such that $G[S]$ is isomorphic to a given graph $H$. We call an edge $uv$, $H$-critical for $S$ if $G/uv[f(S)]$ is non-isomorphic to $H$. Furthermore, we call the edge $uv$ $H$-critical in $G$ if for any vertex subset $S$ that induces $H$, $uv$ is $H$-critical for $S$.

\begin{theorem}{\label{Theorem: constructive edge characterization}}
	Let $G$ be a graph and $S \subseteq V(G)$ where $H$ is the graph induced by $S$ in $G$. For any edge $uv \in E(G)$, $uv$ is $H$-critical for $S$ if and only if 
	\begin{enumerate}
		\item $	u,v \in S$ or
		\item $u \in V(G) \setminus S$, $v \in S$, and $u$ is not a corner dominated by $v$ in the subgraph $G[S\cup \{u\}]$.
	\end{enumerate}
\end{theorem}
\begin{proof}
	\begin{enumerate}
		\item If $u,v \in S$, then $|f(S)| < |S|$. Thus, $G/uv[f(S)]$ is non-isomorphic to $H$.
		\item Let $u \in V(G) \setminus S$, $v \in S$, and $u$ is not a corner dominated by $v$ in the subgraph $G[S\cup \{u\}]$. Additionally, let $w \in N_{S}(u)$ but $w \notin N_{S}(v)$. In $G/uv$, let $x := V(G/uv) - V(G)$. Clearly, $x$ is adjacent to any vertex in $N_{S}(v) \cup \{w\}$. Hence, the size of $G/uv[f(S)]$ is larger than that of $G[S]$. Thus, $G/uv[f(S)]$ is non-isomorphic to $H$.
	\end{enumerate}
	
	Conversely, if none of the conditions in the theorem hold, then one of the following holds:
	\begin{enumerate}
		\item both $u$ and $v$ are not in $S$, or
		\item $u \in V(G) \setminus S$, $v \in S$, and $u$ is a corner dominated by $v$ in the subgraph $G[S\cup \{u\}]$.
	\end{enumerate}
	In both cases, $G[S] \cong G/uv[f(S)] \cong H$. Consequently, $uv$ is not $H$-critical for $S$.
\end{proof}

In the following section, we demonstrate how to use Theorem \ref{Theorem: characeterization} by using it to characterize claw-free graphs and line graphs.

\section{Special graphs}{\label{Section: application}}
\begin{proposition}{\label{Lemma: Cycle contraction}}
	If a graph $G$ is $C_{n}$-exist, where $n \geq 4$, then there is a $G$-contraction that is $C_{n-1}$-exist. 
\end{proposition}
\begin{proposition}{\label{Lemma: EC(C3)}}
	The only critical $C_{3}$-exist graph is $C_{3}$.
\end{proposition}

\subsection{Claw-Free Graphs}
A \emph{claw} is a star $S_{3}$. There are several graph families that are subfamilies of claw-free graphs, for instance, line graphs and complements of triangle-free graphs. For more graph families and results about claw-free graphs, please consult \cite{faudree1997claw}. Additionally, for more structural results about claw-free graphs, please consult \cite{chudnovsky2008claw,chudnovsky2005structure}. In the following, we call the graph $H_{5}$ in Figure \ref{Figure: the graphs in CC(claw)} bull.

\begin{proposition}{\label{Proposition: CC(claw)}}
	The graphs in Figure \ref{Figure: the graphs in CC(claw)} are the only claw-split graphs.
\end{proposition}
Identifying the claw-free graphs  among the graphs presented in Figure \ref{Figure: the graphs in CC(claw)}, we obtain the following.
\begin{corollary}{\label{Corollary: FCC(claw)}}
	Bull is the only claw-free-split graph.
\end{corollary}
\begin{figure}[ht!]
	\centering
	\begin{tikzpicture}[hhh/.style={draw=black,circle,inner sep=2pt,minimum size=0.2cm}]
		\begin{scope}[shift={(-3,0)},scale=1.5]
			\node 	   (h) at (0,0) 	{$H_{1}$};
			\node[hhh] (e) at (0,0.3) 	{};
			\node[hhh] (a) at (0,0.8) 	{};
			\node[hhh] (b) at (0.4,1) 	{};
			\node[hhh] (c) at (0,1.3) 	{};
			\node[hhh] (d) at (-0.4,1) 	{};
			
			\draw (c) -- (a) -- (b)-- (c) -- (d) -- (a) --(e) ;
		\end{scope}
		
		\begin{scope}[shift={(0,0)},scale=1.5]
			\node 	(h) at (0,0) 	{$H_{2}$};
			\node[hhh] 	(e) at (-0.4,0.3) 	{};
			\node[hhh] 	(d) at (0.4,0.3) 	{};
			\node[hhh] 	(a) at (0,0.8) 	{};
			\node[hhh]  (b) at (-0.4,1.3) 	{};
			\node[hhh] 	(c) at (0.4,1.3) 	{};
			
			\draw (a) -- (c) -- (b)-- (a) -- (d)  (e) --(a);
		\end{scope}
		
		\begin{scope}[shift={(3,0)},scale=1.5]
			\node 	(h) at (0,0) 	{$H_{3}$};
			\node[hhh] (d) at (-0.4,0.3) 	{};
			\node[hhh] (e) at (0.4,0.3) 	{};
			\node[hhh] 	(a) at (0,0.8) 	{};
			\node[hhh] (b) at (-0.4,1.3) 	{};
			\node[hhh] (c) at (0.4,1.3) 	{};
			
			\draw (c) -- (a) -- (b)  (d) -- (a) --(e);
		\end{scope}
		
		\begin{scope}[shift={(-3,-3)},scale=1.3]
			\node 	(h) at (0,0) 	{$H_{4}$};
			\node[hhh] 	(a) at (0,0.8) 	{};
			\node[hhh]  (b) at (0,1.3) 	{};
			\node[hhh] 	(c) at (0,1.8) 	{};
			\node[hhh] 	(d) at (-0.4,0.3) 	{};
			\node[hhh] 	(e) at (0.4,0.3) 	{};
			
			\draw (c) -- (b)-- (a) -- (d)  (e) --(a);
		\end{scope}
		
		\begin{scope}[shift={(0,-3)},scale=1.3]
			\node 	(h) at (0,0) 	 {$H_{5}$};
			\node[hhh] 	(a) at (0,1.5) 		{};
			\node[hhh]  (b) at (-0.5,1) 	{};
			\node[hhh] 	(c) at (-0.5,0.3) 	{};
			\node[hhh] 	(d) at (0.5,1) 		{};
			\node[hhh] 	(e) at (0.5,0.3) 	{};
			
			\draw (a) -- (b) --(d) --(a)   (c) -- (b) (d)--(e);
		\end{scope}
		
		\begin{scope}[shift={(3,-3)},scale=1.3]
			\node 	(h) at (0,0) 	{$H_{6}$};
			\node[hhh] 	(a) at (0,1.5) 	{};
			\node[hhh]  (b) at (-0.5,1) 	{};
			\node[hhh] 	(c) at (-0.5,0.3) 	{};
			\node[hhh] 	(d) at (0.5,1) 	{};
			\node[hhh] 	(e) at (0.5,0.3) 	{};
			
			\draw (a) -- (b) --(d) --(a)  (d) -- (c) -- (b) (d) -- (e) -- (b);
		\end{scope}
	\end{tikzpicture}	
	\caption{The claw-split graphs}
	\label{Figure: the graphs in CC(claw)}
\end{figure}
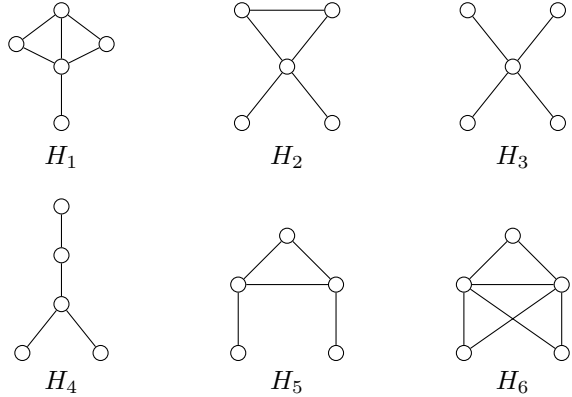

\begin{proposition}{\label{Proposition: EC(claw)}}
	The graphs in Figure \ref{Figure: the graphs in EC(claw)} are the only critically claw-exist graphs.
\end{proposition}
\begin{figure}[ht!]
	\centering
	\begin{tikzpicture}[hhh/.style={draw=black,circle,inner sep=2pt,minimum size=0.2cm}]
		\begin{scope}[shift={(-3.5,0)},scale=2]
			\node 	   (a) at (-0.25,0) 		{$H_{1}$};
			\node[hhh] (b) at (-0.5,0.6)	{};
			\node[hhh] (c1) at (0,0.3) 		{};
			\node[hhh] (c2) at (0,0.6) 		{};
			\node[hhh] (c3) at (0,0.9) 		{};
			
			\draw (c1) -- (b) (c2) -- (b)  (c3) -- (b);
		\end{scope}
		
		\begin{scope}[shift={(-1.5,0)},scale=2]
			\node 	   (a) at (-0.25,0) 		{$H_{2}$};
			\node[hhh] (b1) at (-0.5,0.75)	{};
			\node[hhh] (b2) at (-0.5,0.45)	{};
			\node[hhh] (c1) at (0,0.3) 		{};
			\node[hhh] (c2) at (0,0.6) 		{};
			\node[hhh] (c3) at (0,0.9) 		{};
			
			\draw (c1) -- (b1) (c2) -- (b1)  (c3) -- (b1)
			(c1) -- (b2) (c2) -- (b2)  (c3) -- (b2);
		\end{scope}
		
		\begin{scope}[shift={(0.5,0)},scale=2]
			\node 	   (a) at (-0.25,0) 		{$H_{3}$};
			\node[hhh] (b1) at (-0.5,0.3)	{};
			\node[hhh] (b2) at (-0.5,0.6)	{};
			\node[hhh] (b3) at (-0.5,0.9)	{};
			\node[hhh] (c1) at (0,0.3) 		{};
			\node[hhh] (c2) at (0,0.6) 		{};
			\node[hhh] (c3) at (0,0.9) 		{};
			
			\draw (c1) -- (b1) (c2) -- (b1)  (c3) -- (b1)
			(c1) -- (b2) (c2) -- (b2)  (c3) -- (b2)
			(c1) -- (b3) (c2) -- (b3)  (c3) -- (b3);
		\end{scope}
		
		\begin{scope}[shift={(2.5,0)},scale=2]
			\node 	   (a) at (0,0) 		{$H_{4}$};
			\node[hhh] (b) at (-0.5,0.6)	{};
			\node[hhh] (c1) at (0,0.3) 		{};
			\node[hhh] (c2) at (0,0.6) 		{};
			\node[hhh] (c3) at (0,0.9) 		{};
			\node[hhh] (d) at (0.5,0.75) 	{};
			
			\draw (c1) -- (b) (c2) -- (b)  (c3) -- (b) (c2)--(d)--(c3);
		\end{scope}
		
		\begin{scope}[shift={(-2,-2.5)},scale=2]
			\node 	   (a) at (0,0) 		{$H_{5}$};
			\node[hhh] (b) at (-0.5,0.6)	{};
			\node[hhh] (c1) at (0,0.3) 		{};
			\node[hhh] (c2) at (0,0.6) 		{};
			\node[hhh] (c3) at (0,0.9) 		{};
			\node[hhh] (d1) at (0.5,0.75) 	{};
			\node[hhh] (d2) at (0.5,0.45) 	{};
			
			\draw (c1) -- (b) (c2) -- (b)  (c3) -- (b) 
			(c2)--(d1)--(c3) (c1)--(d2)--(c2);
		\end{scope}
		
		\begin{scope}[shift={(1,-2.5)},scale=2]
			\node 	   (a) at (0,0) 		{$H_{6}$};
			\node[hhh] (b1) at (-0.5,0.75)	{};
			\node[hhh] (b2) at (-0.5,0.45) 	{};
			\node[hhh] (c1) at (0,0.3) 		{};
			\node[hhh] (c2) at (0,0.6) 		{};
			\node[hhh] (c3) at (0,0.9) 		{};
			\node[hhh] (d) at (0.5,0.75) 	{};

			\draw (c1) -- (b1) (c2) -- (b1)  (c3) -- (b1) 
			(c1) -- (b2) (c2) -- (b2)  (c3) -- (b2) 
			(c2)--(d)--(c3);
		\end{scope}
	\end{tikzpicture}	
	\caption{The critically claw-exist graphs}
	\label{Figure: the graphs in EC(claw)}
\end{figure}
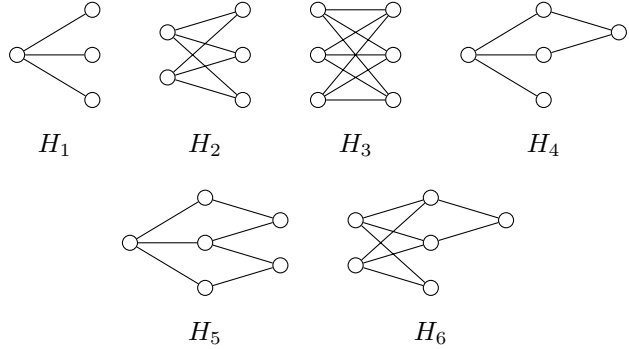

\begin{proof}
	\renewcommand{\qedsymbol}{\claimqed}
	Through this proof, we assume that $G$ is a critically claw-exist graph with $S := \{r,s,t,u\}$, where $G[S]$ is isomorphic to a claw and $u$ is its center. By Theorem \ref{Theorem: critical H-exist: not S is independent}, $V(G) \setminus S$ is independent. Thus, any vertex in $V(G) \setminus S$ is adjacent to vertices only in $S$. By Corollary \ref{coro: no vertex adajcet to 1 or 2 or 3 or max degree in critical graph}, if $v \in V(G) \setminus S$, then neither $|N(v)| = 1$ nor $v$ is adjacent to $u$.
	
	Let $v,w \in V(G) \setminus S$ such that $N(v) = N(w)$ where $|N(v)|=2$.
	W.l.o.g., assume that $N(v) =\{r,s\}$, however, $G/tu[\{r,u,v,w\}]$ induces a claw, which contradicts the fact that $G$ is a critically claw-exist. Thus, if $v,w \in V(G) \setminus S$ such that $|N(v)| = |N(w)|=2$, then $N(v) \not= N(w)$.
	
	Let $v,w,x \in V(G) \setminus S$ such that $|N(v)| = |N(w)| = 2$.
	No two vertices of $v,w$, and $x$ (if $|N(x)| =2$) are adjacent to the same vertices in $S$. W.l.o.g., assume that $N(v) =\{r,s\}$, $N(w) =\{r,t\}$, and $ \{s,t\} \subseteq N(x)$. The graph $G/tx[\{r,u,v,w\}]$ induces a claw, which contradicts the fact that $G$ is a critically claw-exist. Thus, if $v,w \in V(G) \setminus S$ such that $|N(v)| = |N(w)|=2$, then $G$ is isomorphic to $H_{5}$ in Figure \ref{Figure: the graphs in EC(claw)}.

	Let $v,w,x \in V(G) \setminus S$ such that $|N(v)| = |N(w)| = 3$. 
	W.l.o.g., assume that $s$ is adjacent to $x$. The graph $G/sx[\{r,u,v,w\}]$ induces a claw, which contradicts the fact that $G$ is a critically claw-exist. Thus, if $v,w \in V(G) \setminus S$ such that $N(v) = N(w)$ and $|N(v)|=3$, then $G$ is isomorphic to $H_{3}$ in Figure \ref{Figure: the graphs in EC(claw)}.

	Consequently, the possible critically claw-exist graphs are those presented in Figure \ref{Figure: the graphs in EC(claw)}. To complete the proof, we have to show that all these graphs are critically claw-exist, which is straightforward in each case.
	\renewcommand{\qedsymbol}{$\square$}
\end{proof}

By Theorem \ref{Theorem: characeterization}, Corollary \ref{Corollary: FCC(claw)}, and Proposition \ref{Proposition: EC(claw)}, we obtain the following result.
\begin{theorem}{\label{Theorem: claw-free characterization}}
	Let $G$ be a bull-free graph that is non-isomorphic to any graph in Figure \ref{Figure: the graphs in EC(claw)}. The graph $G$ is claw-free if and only if any $G$-contraction is claw-free.
\end{theorem}

\subsection{Line Graphs}
The \emph{line} graph of a graph $G$, denoted by $L(G)$, is the graph whose vertex set is the edge set of $G$ and two vertices in $L(G)$ are adjacent if and only if the corresponding edges in $G$ intersect in a vertex. Line graphs were characterized as follows.
\begin{theorem}[\cite{krausz1943demonstration,harary}]\label{lineCharcaterization by harary}
	A graph $G$ is a line graph if and only if its edges can be partitioned into cliques where every vertex of $G$ belong to at most two of such cliques.
\end{theorem}
\begin{theorem}[\cite{beineke1970characterizations}]
	A graph $G$ is a line graph if and only if $G$ is $\{L_{1},\dots,L_{9}\}$-free.
\end{theorem}
We call a graph that is $\{L_{1},\dots,L_{9}\}$-exist \emph{non-line}.

\begin{figure}[ht!]
	\centering
	\begin{tikzpicture}[hhh/.style={draw=black,circle,inner sep=2pt,minimum size=0.2cm},scale=0.8]
		\begin{scope}[shift={(0,0)}]
			\node 	(graph) at (0:0cm) 	{$L_{1}$};
			\node[hhh]  (a) at (0,1) 		{};
			\node[hhh] 	(b) at (0,2) 	{};
			\node[hhh] 	(c) at (-1,0.5) 	{};
			\node[hhh] 	(d) at (1,0.5) {};
			
			\draw (b)--(a)--(c) (a)--(d)  ;
		\end{scope}	
		
		\begin{scope}[shift={(3,0)}]
			\node 	(graph) at (0:0cm) 	{$L_{2}$};
			\node[hhh]  (a) at (0,2.5) 	{};
			\node[hhh] 	(b) at (0,1.5) 	{};
			\node[hhh] 	(c) at (0,0.5)	{};
			\node[hhh] 	(d) at (-1,1) {};
			\node[hhh] 	(e) at (1,1) {};
			
			\draw (e)--(a)--(d) (e)--(b)--(d) (e)--(c)--(d) (c)--(b);
		\end{scope}
		
		\begin{scope}[shift={(6,0)}]
			\node 	(graph) at (0:0cm) 	{$L_{3}$};
			\node[hhh]  (a) at (-1,2.2) 	{};
			\node[hhh]  (f) at (1,2.2) 	{};
			\node[hhh] 	(b) at (0,1.5) 	{};
			\node[hhh] 	(c) at (0,0.5)	{};
			\node[hhh] 	(d) at (-1,1) {};
			\node[hhh] 	(e) at (1,1) {};
			
			\draw (e)--(f)--(a)--(d) (e)--(b)--(d) (e)--(c)--(d) (c)--(b);
		\end{scope}
		
		\begin{scope}[shift={(9,0)}]
			\node 	(graph) at (0:0cm) 	{$L_{4}$};
			\node[hhh]  (a) at (-1,2) {};
			\node[hhh]  (f) at (1,2) 	{};
			\node[hhh] 	(b) at (0,1.5) 	{};
			\node[hhh] 	(c) at (0,0.5)	{};
			\node[hhh] 	(d) at (-1,1) {};
			\node[hhh] 	(e) at (1,1) {};
			
			\draw (e)--(f) (a)--(d) (e)--(b)--(d) (e)--(c)--(d) (c)--(b);
		\end{scope}
		
		\begin{scope}[shift={(12,0)}]
			\node 	(graph) at (0:0cm) 	{$L_{5}$};
			\node[hhh]  (a) at (0,3.1) 	{};
			\node[hhh] 	(b) at (0,2.3) 	{};
			\node[hhh] 	(c) at (-1,1.5)	{};
			\node[hhh] 	(d) at (1,1.5) {};
			\node[hhh] 	(e) at (1,0.5) {};
			\node[hhh]  (f) at (-1,0.5) 	{};

			\draw (a)--(b)--(c)--(d)--(b) (d)--(e)--(f)--(c)--(e) (d)--(f);
		\end{scope}
		
		\begin{scope}[shift={(1.5,-4)}]
			\node 	(graph) at (0:0cm) 	{$L_{6}$};
			\node[hhh]  (a) at (1,2.5) 	{};
			\node[hhh] 	(b) at (-1,2.5) {};
			\node[hhh] 	(c) at (-1,1.5)	{};
			\node[hhh] 	(d) at (1,1.5) {};
			\node[hhh] 	(e) at (1,0.5) {};
			\node[hhh]  (f) at (-1,0.5) {};

			\draw (a)--(b)--(c)--(d)--(a)--(c) (d)--(b) (d)--(e)--(f)--(c)--(e) (d)--(f);
		\end{scope}
		
		\begin{scope}[shift={(4.5,-4)}]
			\node 	(graph) at (0:0cm) 	{$L_{7}$};
			\node[hhh]  (a) at (1,2.5) 	{};
			\node[hhh] 	(b) at (-1,2.5) {};
			\node[hhh] 	(c) at (-1,1.5)	{};
			\node[hhh] 	(d) at (1,1.5) {};
			\node[hhh] 	(e) at (1,0.5) {};
			\node[hhh]  (f) at (-1,0.5) {};

			\draw (a)--(b)--(c)--(d)--(a)--(c) (d)--(e)--(f) (d)--(f)--(c);
		\end{scope}
		
		\begin{scope}[shift={(7.5,-4)}]
			\node 	(graph) at (0:0cm) 	{$L_{8}$};
			\node[hhh]  (a) at (0,2) 	{};
			\node[hhh] 	(b) at (-1,0.5) {};
			\node[hhh] 	(c) at (1,0.5) 	{};
			\node[hhh] 	(d) at (0,1.1) 	{};
			\node[hhh] 	(e) at (0,3) 	{};
			
			\draw (a)--(b)--(c)--(a)--(d)--(b) (d)--(c)--(e)--(a)--(e)--(b);
		\end{scope}
		
		\begin{scope}[shift={(10.5,-4)}]
			\node 	(graph) at (0:0cm) 	{$L_{9}$};
			\node[hhh]  (a) at (0,1.5) 	{};
			\node[hhh] 	(b) at (0,2.5) 	{};
			\node[hhh] 	(c) at (-1,1.7)	{};
			\node[hhh] 	(d) at (-0.7,0.5) {};
			\node[hhh] 	(e) at (0.7,0.5) {};
			\node[hhh]  (f) at (1,1.7) 	{};

			\draw (b)--(c)--(d)--(e)--(f)--(b)--(a)--(c) (d)--(a)--(e) (a)--(f);
		\end{scope}
	\end{tikzpicture}	
	\caption{The forbidden subgraphs in line graphs}
	\label{fig: Forbidden graphs in line graphs}
\end{figure}
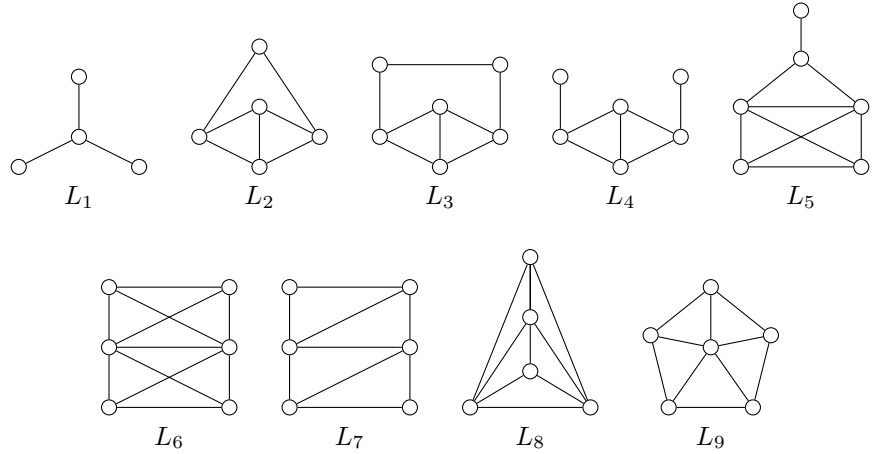

\begin{figure}[ht!]
	\centering
	\begin{tikzpicture}[hhh/.style={draw=black,circle,inner sep=2pt,minimum size=0.2cm}]
		
		\begin{scope}[shift={(0,0)}]
			\node 	   (a) at (0,0) 		{$L_{10}$};
			\node[hhh] (b1) at (-0.5,0.7)	{};
			\node[hhh] (b2) at (-0.5,1.7)	{};
			\node[hhh] (c1) at (0.5,0.5) 		{};
			\node[hhh] (c2) at (0.5,1.2) 		{};
			\node[hhh] (c3) at (0.5,2) 		{};
			
			\draw (c1) -- (b1) (c2) -- (b1)  (c3) -- (b1)
			(c1) -- (b2) (c2) -- (b2)  (c3) -- (b2);
		\end{scope}
		
		\begin{scope}[shift={(2,0)}]
			\node 	   (a) at (0,0) 		{$L_{11}$};
			\node[hhh] (b1) at (-0.5,0.5)	{};
			\node[hhh] (b2) at (-0.5,1.2)	{};
			\node[hhh] (b3) at (-0.5,2)	{};
			\node[hhh] (c1) at (0.5,0.5) 		{};
			\node[hhh] (c2) at (0.5,1.2) 		{};
			\node[hhh] (c3) at (0.5,2) 		{};
			
			\draw (c1) -- (b1) (c2) -- (b1)  (c3) -- (b1)
			(c1) -- (b2) (c2) -- (b2)  (c3) -- (b2)
			(c1) -- (b3) (c2) -- (b3)  (c3) -- (b3);
		\end{scope}
		
		\begin{scope}[shift={(4.5,0)}]
			\node 	   (a) at (0,0) 		{$L_{12}$};
			\node[hhh] (b) at (-1,1.2)	{};
			\node[hhh] (c1) at (0,0.5) 		{};
			\node[hhh] (c2) at (0,1.2) 		{};
			\node[hhh] (c3) at (0,1.9) 		{};
			\node[hhh] (d) at (1,1.5) 	{};
			
			\draw (c1) -- (b) (c2) -- (b)  (c3) -- (b) (c2)--(d)--(c3);
		\end{scope}
		
		\begin{scope}[shift={(7.5,0)}]
			\node 	   (a) at (0,0) 		{$L_{13}$};
			\node[hhh] (b) at (-1,1.2)	{};
			\node[hhh] (c1) at (0,0.5) 		{};
			\node[hhh] (c2) at (0,1.2) 		{};
			\node[hhh] (c3) at (0,1.9) 		{};
			\node[hhh] (d1) at (1,1.5) 	{};
			\node[hhh] (d2) at (1,0.8) 	{};
			
			\draw (c1) -- (b) (c2) -- (b)  (c3) -- (b) 
			(c2)--(d1)--(c3) (c1)--(d2)--(c2);
		\end{scope}
		
	\end{tikzpicture}	
	\caption{Critically non-line graphs}
	\label{fig: Extra critical line-exist graphs}
\end{figure}
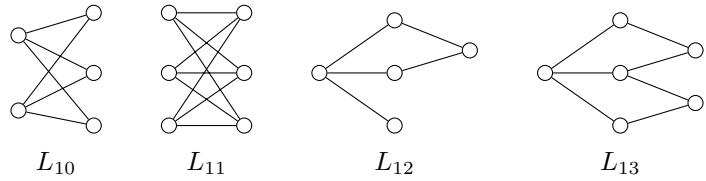


\subsubsection{Critically Non-Line Graphs}

We present the following proposition, which will be useful later.
\begin{proposition}\label{pro: claw adn a vertex outside with adajcent vertices inside}
	Let $G$ be a critically non-line graph where $\{r,s,t,u,v\} \subset V(G)$ such that $\{r,s,t\}$ induces $P_{3}$ in $G$ and $rs,st,uv \in E(G)$. If $w \in V(G) \setminus \{r,s,t,u,v\}$ is adjacent to $s$, then $w$ is adjacent to $r$ or $t$.
\end{proposition}
\begin{proof}
	For the sake of contradiction, assume that $w \in V(G) \setminus \{r,s,t,u,v\}$ is adjacent to $s$ but neither $r$ nor $t$. However, $\{r,s,t,w\}$ is a claw while $u,v$ are two adjacent vertices outside the claw, which is a contradiction to Theorem \ref{Theorem: critical H-exist: not S is independent}.
\end{proof}

\subsubsection{The critically non-line graphs that are \texorpdfstring{$L_{1}$}{L1}-exist}
\begin{lemma}\label{lem: L1 critical line-exist}
	The graphs $L_{1},L_{10},\dots,L_{13}$ are the only critically non-line graphs that are $L_{1}$-exist.
\end{lemma}
\begin{proof}
	By Proposition \ref{Proposition: EC(claw)}, it was proved that the graphs in Figure \ref{Figure: the graphs in EC(claw)} are the graphs that are critically claw-exist. Hence, we identify which of these graphs are critically non-line.
	The graph $L_{1}$ is trivially critically non-line. Any $L_{10}$-contraction and $L_{12}$-contraction has independence number $2$ and order $4$. Thus, any $L_{10}$-contraction and $L_{12}$-contraction is claw-free. Thus, $L_{10}$ and $L_{12}$ are critically non-line. Any $L_{11}$-contraction has independence number $2$, order $5$, and maximum degree $4$. Thus, any $L_{11}$-contraction is claw-free and $L_{2}$-free. Thus, $L_{11}$ is a critically non-line. Any $L_{13}$-contraction has independence number $2$ and, order $5$, and either maximum degree $4$ or three vertices of degree $2$. Thus, any $L_{13}$-contraction is $L_{1}$-free and $L_{2}$-free. Thus, $L_{13}$ is a critically non-line. Finally, we note that contracting an edge incident to the vertex of degree $2$ in graph $H_{6}$ in Figure \ref{Figure: the graphs in EC(claw)} construct a graph that is isomorphic to $L_{2}$. Thus, that graph is not critically non-line.
\end{proof}

\subsubsection{The critically non-line graphs that are $L_{2}$-exist}
\begin{lemma}\label{lem: L2 critical line-exist}
	The graph $L_{2}$ is the only critically non-line graph that is $L_{2}$-exist.
\end{lemma}
\begin{proof}
	For the sake of contradiction, we assume that there is a critically non-line $G$ that is non-isomorphic to $L_{2}$ where $G$ has a vertex set $S:=\{r,s,t,u,v\}$ that induces $L_{2}$ as in the following figure on right.

	\begin{minipage}{0.75\textwidth}
		Let $w \in V(G) \setminus S$. By Theorem \ref{Theorem: critical H-exist: not S is independent}, we note that $V(G) \setminus S$ is independent. By Corollary \ref{coro: no vertex adajcet to 1 or 2 or 3 or max degree in critical graph}, the neighborhood of $w$ is neither one vertex nor a set of two adjacent vertices. If the neighborhood of $w$ is two nonadjacent vertices, then $w$ is adjacent to a non leaf vertex in a $P_{3}$, but not to a leaf of the path, which contradicts Proposition \ref{pro: claw adn a vertex outside with adajcent vertices inside}. 	
	\end{minipage}\hfill
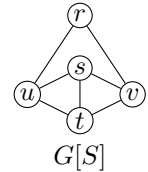
	\begin{minipage}{0.2\textwidth}
		
		\centering
		\begin{tikzpicture}[hhh/.style={draw=black,circle,inner sep=1pt,minimum size=0.1cm},scale=0.7]
			\node[hhh]  (r) at (0,2.5) 	{$r$};
			\node[hhh] 	(s) at (0,1.5) 	{$s$};
			\node[hhh] 	(t) at (0,0.5)	{$t$};
			\node[hhh] 	(u) at (-1,1) {$u$};
			\node[hhh] 	(v) at (1,1) {$v$};
			
			\node 	(g) at (0,-0.2) {$G[S]$};
			
			\draw (v)--(r)--(u) (v)--(s)--(u) (v)--(t)--(u) (t)--(s);
		\end{tikzpicture}
	\end{minipage}
	By Corollary \ref{coro: no vertex adajcet to 1 or 2 or 3 or max degree in critical graph}, the neighborhood of $w$ is not three vertices that induce either $P_{3}$ or $C_{3}$. If $N(w)=\{r,s,t\}$, then $w$ is adjacent to a nonleaf vertex $r$ while $\{r,u,v\}$ induces $P_{3}$ and $s,t$ are two adjacent vertices, which contradicts Proposition \ref{pro: claw adn a vertex outside with adajcent vertices inside}. If $N(w)=S \setminus \{u\}$, $N(w)=S \setminus \{v\}$, or $N(w)=S \setminus \{r\}$, then $w$ is a corner dominated by a vertex in $S$ while $S$ induces $L_{2}$, which is a contradiction to Theorem \ref{Theorem: critical H-exist: not S is independent}. If $\{r,s,u,v\} \subseteq N(w)$ (or similarly $\{r,t,u,v\} \subseteq N(w)$), then in $G/tu$ the set $S \cup\{w\}$ induces $L_{8}$, which is a contradiction to the assumption that $G$ is a critically non-line graph. Finally, we prove that $L_{2}$ is a critically non-line graph to complete the proof. Any $L_{2}$-contraction has independence number at most $2$ and order $4$. Thus, any $L_{2}$-contraction is claw-free. Thus, $L_{2}$ is a critically non-line graph.
\end{proof}

\subsubsection{The critically non-line graphs that are $L_{3}$-exist}
\begin{lemma}\label{lem: L3 is not critical line-exist}
	There is no critically non-line graph that is $L_{3}$-exist.
\end{lemma}
\begin{proof}
	For the sake of contradiction assume that there is a critically non-line graph $G$ that is $L_{3}$-exist. Let $S$ be a vertex set in $G$ that induces a graph isomorphic to $L_{3}$ where $u,v \in S$ such that $deg_{S}(u)= deg_{S}(v)=2$. However, $G/uv[S]$ induces a graph isomorphic to $L_{2}$, which contradicts the assumption that $G$ is a critically non-line.
\end{proof}

\subsubsection{The critically non-line graphs that are $L_{4}$-exist}
\begin{lemma}\label{lem: L4 critical line-exist}
	The graph $L_{4}$ is the only critically non-line graph that is $L_{4}$-exist.
\end{lemma}
\begin{proof}
	For the sake of contradiction, we assume that there is a critically non-line graph $G$ that is non-isomorphic to $L_{4}$ where $G$ has a vertex set $S$ that induce $L_{4}$.

	\begin{minipage}{0.75\textwidth}
		Assume that $S:=\{r,s,t,u,v,w\}$ as presented in the following figure on right. Let $x \in V(G) \setminus S$. By Theorem \ref{Theorem: critical H-exist: not S is independent}, we note that $V(G) \setminus S$ is independent. By Corollary \ref{coro: no vertex adajcet to 1 or 2 or 3 or max degree in critical graph}, the neighborhood of $x$ is not one vertex in $S$. According to the cardinality of the neighborhood of $x$ we partition the proof to cases.
	\end{minipage}\hfill
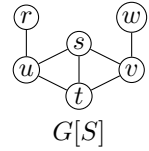
	\begin{minipage}{0.2\textwidth}
		
		\centering
		\begin{tikzpicture}[hhh/.style={draw=black,circle,inner sep=1pt,minimum size=0.2cm},scale=0.7]
			\node[hhh]  (r) at (-1,2) {$r$};
			\node[hhh]  (w) at (1,2) 	{$w$};
			\node[hhh] 	(s) at (0,1.5) 	{$s$};
			\node[hhh] 	(t) at (0,0.5)	{$t$};
			\node[hhh] 	(u) at (-1,1) {$u$};
			\node[hhh] 	(v) at (1,1) {$v$};
			
			\node 	(g) at (0,-0.2) {$G[S]$};
			
			\draw (v)--(w) (r)--(u) (v)--(s)--(u) (v)--(t)--(u) (t)--(s);
		\end{tikzpicture}
	\end{minipage}
	
	Case $1$: The cardinality of the neighborhood of $x$ is $2$. By Corollary \ref{coro: no vertex adajcet to 1 or 2 or 3 or max degree in critical graph}, $x$ is nonadjacent to two adjacent vertices. If $N(x)=\{r,w\}$, then $G/rx[S \cup \{x\}]$ induces $L_{3}$, which is a contradiction to the fact that $G$ is a critically non-line graph. If $x$ is adjacent to $u$ (or similarly to $v$), then $\{r,s,u,x\}$ induces a claw while $t$ is a corner dominated by $s$ which is a contradiction to Theorem \ref{Theorem: critical H-exist: not S is independent}. Otherwise, $x$ is adjacent to $s$ (or similarly $t$) in which $\{s,u,v,x\}$ induces a claw while $t$ is a corner dominated by $s$ which is a contradiction to Theorem \ref{Theorem: critical H-exist: not S is independent}. Thus, $|N(x)| \not= 2$.

	Case 2:	The cardinality of the neighborhood of $x$ is $3$. By Corollary \ref{coro: no vertex adajcet to 1 or 2 or 3 or max degree in critical graph}, $x$ is nonadjacent to three vertices that induce either $P_{3}$ or $C_{3}$. If $N(x)=\{r,s,w\}$ (or similarly $\{r,t,w\}$), then $\{r,s,w,x\}$ induces a claw while $t$ is a corner dominated by $s$, which is a contradiction to Theorem \ref{Theorem: critical H-exist: not S is independent}. If $N(x)=\{r,u,w\}$ (or similarly $N(x)=\{r,v,w\}$), then $\{x\} \cup S \setminus \{r\}$ induces $L_{3}$ which is a contradiction to Lemma \ref{lem: L3 is not critical line-exist}. If $N(x)=\{r,u,v\}$ (or similarly $N(x)=\{u,v,w\}$), then $\{s,v,w,x\}$ induces a claw while $t$ is a corner dominated by $s$, which is a contradiction to Theorem \ref{Theorem: critical H-exist: not S is independent}. if $N(x)=\{r,s,v\}$ (or similarly $\{r,t,v\},\{s,u,w\},\{t,u,w\}$), then $G/ru[\{t,v,w,x\}]$ induces a claw, which is a contradiction to the fact that $G$ is a critically non-line graph. 
	Finally, if $N(x)=\{r,s,t\}$ (or similarly $N(x)=\{s,t,w\}$), then in $G/vw$, $\{t,u,v,x\}$ induces a claw, which is a contradiction to the assumption that $G$ is a critically non-line. Thus, $|N(x)| \not= 3$.

	Case 3:	The cardinality of the neighborhood of $x$ is at least $4$.
	The neighborhood of $x$ is neither $\{r,s,t,u\}$ nor $\{s,t,v,w\}$, otherwise $x$ is a corner dominated by $u$ or $v$ respectively, which contradicts Theorem \ref{Theorem: critical H-exist: not S is independent}. If $\{r,s,w\} \subseteq N(x)$ (or similarly $\{r,t,w\} \subseteq N(x)$), then $\{r,s,w,x\}$ induces a claw while $t$ is a corner dominated by $s$, which is a contradiction to Theorem \ref{Theorem: critical H-exist: not S is independent}. If $N(x)=\{r,u,v,w\}$, then $\{x\} \cup S \setminus \{r,w\}$) induces $L_{2}$ while $r$ is a corner dominated by $x$, which is a contradiction to Theorem \ref{Theorem: critical H-exist: not S is independent}. If $N(x)=\{r,s,u,v\}$ (similarly $N(x)=\{r,t,u,v\}$, $N(x)=\{s,u,v,w\}$, or $N(x)=\{t,u,v,w\}$), then $\{t,v,w,x\}$ induces a claw while $s$ is a corner dominated by $v$, which is a contradiction to Theorem \ref{Theorem: critical H-exist: not S is independent}. If $N(x)=\{r,s,t,v\}$ (or similarly $N(x)=\{s,t,u,w\}$), then in $G/vx$, $\{r,s,t,u,x\}$ induces $L_{2}$, which is a contradiction to the assumption that $G$ is a critically non-line graph.
	If $S \setminus \{r,w\} \subseteq N(x)$, then $\{x\} \cup S \setminus \{r,w\}$ induces $L_{8}$ while $r$ is a corner dominated by $u$, which is a contradiction to Theorem \ref{Theorem: critical H-exist: not S is independent}. Consequently, $|N(x)| \not\geq 4$.
	
	Finally, we prove that $L_{4}$ is a critically non-line to complete the proof. Any $L_{4}$-contraction has a vertex with degree $1$, its order is $5$, independence number at most $3$ and no vertex with degree three has more than two independent vertices in its neighborhood. Hence, any $L_{4}$-contraction is $\{L_{1},L_{2}\}$-free. Thus, $L_{4}$ is a critically non-line graph.
\end{proof}

\subsubsection{The critically non-line graphs that are $L_{5}$-exist}
\begin{lemma}\label{lem: L5 critical line-exist}
	The graph $L_{5}$ is the only critically non-line graph that is $L_{5}$-exist.
\end{lemma}
\begin{proof}
	For the sake of contradiction, we assume that there is a critically non-line graph $G$ that is non-isomorphic to $L_{5}$ where $G$ has a vertex set $S$ that induces $L_{5}$. 
	
	\begin{minipage}{0.75\textwidth}
		Assume that $S:=\{r,s,t,u,v,w\}$ that induces $L_{5}$ as in the following figure on right and $x \in V(G) \setminus S$. By Theorem \ref{Theorem: critical H-exist: not S is independent}, we note that $V(G) \setminus S$ is independent. By Corollary \ref{coro: no vertex adajcet to 1 or 2 or 3 or max degree in critical graph}, the neighborhood of $x$ is not one vertex in $S$. According to the cardinality of the neighborhood of $x$ we partition the proof to cases.	
	\end{minipage}\hfill
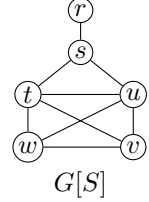
	\begin{minipage}{0.2\textwidth}
		
		\centering
		\begin{tikzpicture}[hhh/.style={draw=black,circle,inner sep=1pt,minimum size=0.2cm},scale=0.7]
			\node[hhh]  (r) at (0,3.1) 	{$r$};
			\node[hhh] 	(s) at (0,2.3) 	{$s$};
			\node[hhh] 	(t) at (-1,1.5)	{$t$};
			\node[hhh] 	(u) at (1,1.5) {$u$};
			\node[hhh] 	(v) at (1,0.5) {$v$};
			\node[hhh]  (w) at (-1,0.5) 	{$w$};
			
			\node 	(g) at (0,-0.2) {$G[S]$};
			\draw (r)--(s)--(t)--(u)--(s) (u)--(v)--(w)--(t)--(v) (u)--(w);
		\end{tikzpicture}
	\end{minipage}
	
	Case $1$: The cardinality of the neighborhood of $x$ is $2$. By Corollary \ref{coro: no vertex adajcet to 1 or 2 or 3 or max degree in critical graph}, $x$ is nonadjacent to two adjacent vertices. 
	If $N(w)=\{r,w\}$ (or similarly $N(w)=\{r,v\}$), then $\{x\}\cup S \setminus \{v\}$ induces $L_{3}$ which is a contradiction to Lemma \ref{lem: L3 is not critical line-exist}. 
	Otherwise, $x$ is adjacent to a vertex with two independent neighbors where $x$ is adjacent to neither, which contradicts Proposition \ref{pro: claw adn a vertex outside with adajcent vertices inside}. Consequently, $|N(x)| \not= 2$.
	
	Case $2$: The cardinality of the neighborhood of $x$ is $3$.
	If $N(x)=\{r,s,w\}$ (or similarly $\{r,s,v\}$), then $\{x\}\cup S \setminus \{r,v\}$ induces $L_{2}$ while $v$ is a corner dominated by $t$ which is a contradiction to Theorem \ref{Theorem: critical H-exist: not S is independent}. 
	If $N(x)=\{r,t,u\}$, then $\{s,t,w,x\}$) induces a claw while $u$ is a corner dominated by $t$, which is a contradiction to Theorem \ref{Theorem: critical H-exist: not S is independent}. 
	If $N(x)=\{r,t,v\}$ (similarly $N(x)=\{r,t,w\}$, $N(x)=\{r,u,v\}$, or $N(x)=\{r,u,w\}$), then $\{s,t,w,x\}$ induces a claw while $u$ is a corner dominated by $t$, which is a contradiction to Theorem \ref{Theorem: critical H-exist: not S is independent}. 
	If $N(x)=\{r,v,w\}$, then $\{x\}\cup S \setminus \{w\}$ induces $L_{3}$ which is a contradiction to Lemma \ref{lem: L3 is not critical line-exist}. 
	Finally, if $N(x)=\{s,v,w\}$, then $\{x\}\cup S \setminus \{r,w\}$ induces $L_{2}$ while $w$ is a corner dominated by $u$ which is a contradiction to Theorem \ref{Theorem: critical H-exist: not S is independent}. Consequently, $|N(x)| \not= 3$.
	
	Case $3$: The cardinality of the neighborhood of $x$ is $4$.
	If $N(x)=\{r,s,t,u\}$ (similarly $\{s,t,u,v\}$, $\{s,t,u,w\}$, $\{s,t,v,w\}$, $\{s,u,v,w\}$, or $\{t,u,v,w\}$), then $x$ is a corner dominated by a vertex in its neighborhood, which contradicts Theorem \ref{Theorem: critical H-exist: not S is independent}. 
	If $N(x)=\{r,s,t,w\}$ (similarly $\{r,s,t,v\}$, $\{r,s,u,w\}$, or $\{r,s,u,v\}$), then $\{x\} \cup S \setminus \{u\}$ induces $L_{7}$ while $u$ is a corner dominated by $t$, which is a contradiction to Theorem \ref{Theorem: critical H-exist: not S is independent}. 
	If $N(x)=\{r,t,u,v\}$ (or similarly $N(x)=\{r,t,u,w\}$), then in $G/xv$ the set $\{x\} \cup S \setminus \{w\}$ induces $L_{2}$ which is a contradiction to the assumption that $G$ is a critically non-line graph. if $N(x)=\{r,s,v,w\}$, then in $G/rs$ the set $\{x\} \cup S \setminus \{w\}$ induces $L_{2}$ which is a contradiction to the assumption that $G$ is a critically non-line.
	Finally, if $N(x)=\{r,u,v,w\}$ (or similarly $\{r,t,v,w\}$), then in $G/xv$ the set $\{x\} \cup S \setminus \{w\}$ induces $L_{2}$ which is a contradiction to the assumption that $G$ is a critically non-line. Consequently, $|N(x)| \not= 4$.
	
	Case $4$: The cardinality of the neighborhood of $x$ is at least $5$. 
	If $N(x)=S \setminus \{r\}$, then $S$ induces $L_{5}$ while $x$ is a corner dominated by $t$, which is a contradiction to Theorem \ref{Theorem: critical H-exist: not S is independent}. 
	If $x$ is nonadjacent to $s$, then $\{x\} \cup S \setminus \{v,w\}$ induces $L_{2}$ while $w$ is a corner dominated by $x$, which is a contradiction to Theorem \ref{Theorem: critical H-exist: not S is independent}. 
	If $\{r,s\} \subseteq N(x)$ and $|N(x)| = 5$, then $\{x\} \cup S \setminus \{r,s\}$ induces $L_{8}$ while $r$ is a corner dominated by $x$, which is a contradiction to Theorem \ref{Theorem: critical H-exist: not S is independent}. Otherwise, $\{x\} \cup S \setminus \{r,v\}$ induces $L_{8}$ while $r$ is a corner dominated by $x$, which is a contradiction to Theorem \ref{Theorem: critical H-exist: not S is independent}.
	Consequently, $|N(x)| \not\geq 5$.
	
	Finally, we prove that $L_{5}$ is a critically non-line to complete the proof. Any $L_{5}$-contraction has independence number at most $2$ and a vertex with degree $1$ or an induced subgraph isomorphic to $K_{4}$. Hence, any $L_{5}$-contraction is $\{L_{1},L_{2}\}$-free. Thus, $L_{5}$ is a critically non-line graph. 
\end{proof}

\subsubsection{The critically non-line graphs that are $L_{6}$-exist}
\begin{lemma}\label{lem: L6 critical line-exist}
	The graph $L_{6}$ is the only critically non-line graph that is $L_{6}$-exist.
\end{lemma}
\begin{proof}
	For the sake of contradiction, we assume that there is a critically non-line graph $G$ that is non-isomorphic to $L_{6}$ where $G$ has a vertex set $S:=\{r,s,t,u,v,w\}$ that induces $L_{6}$.

	\begin{minipage}{0.75\textwidth}
		Assume that $G[S]$ is labeled as in the following figure on right. By Theorem \ref{Theorem: critical H-exist: not S is independent}, we note that $V(G) \setminus S$ is independent. By Corollary \ref{coro: no vertex adajcet to 1 or 2 or 3 or max degree in critical graph}, there is no vertex $ x \in V(G) \setminus S$ whose neighborhood is one vertex in $S$, or two adjacent vertices in $S$. Moreover, $x$ is nonadjacent to $t$ or $u$.		
	\end{minipage}\hfill
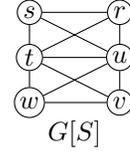
	\begin{minipage}{0.2\textwidth}
		
		\centering
		\begin{tikzpicture}[hhh/.style={draw=black,circle,inner sep=1pt,minimum size=0.2cm},scale=0.6]
			\node[hhh]  (r) at (1,2.5) 	{$r$};
			\node[hhh] 	(s) at (-1,2.5) {$s$};
			\node[hhh] 	(t) at (-1,1.5)	{$t$};
			\node[hhh] 	(u) at (1,1.5) {$u$};
			\node[hhh] 	(v) at (1,0.5) {$v$};
			\node[hhh]  (w) at (-1,0.5) {$w$};
			
			\node 	(g) at (0,-0.2) {$G[S]$};
			\draw (r)--(s)--(t)--(u)--(r)--(t) (u)--(s) (u)--(v)--(w)--(t)--(v) (u)--(w);
		\end{tikzpicture}
	\end{minipage}
	If $x$ is adjacent to two independent vertices say $r,v$, then $\{r,t,u,v,x\}$ induces a graph isomorphic to $L_{2}$ while $w$ is a corner dominated by $v$, which contradicts Theorem \ref{Theorem: critical H-exist: not S is independent}.
	We complete the proof by showing that $L_{6}$ is a critically non-line.
	Any $L_{6}$-contraction has independence number $2$, order $5$, and a vertex with degree $4$. Hence any $L_{6}$-contraction is $\{L_{1},L_{2}\}$-free. Thus, $L_{6}$ is a critically non-line.
\end{proof}

\subsubsection{The critically non-line graphs that are $L_{7}$-exist}
\begin{lemma}\label{lem: L7 critical line-exist}
	The graph $L_{7}$ is the only critically non-line graph that is $L_{7}$-exist.
\end{lemma}
\begin{proof}
	For the sake of contradiction, we assume that there is a critically non-line graph $G$ that is non-isomorphic to $L_{7}$ where $G$ has a vertex set $S$ that induce $L_{7}$.

	\begin{minipage}{0.75\textwidth}
		Assume that $G[S]$ is labeled as in the figure on the right and $x \in V(G) \setminus S$. By Theorem \ref{Theorem: critical H-exist: not S is independent}, we note that $V(G) \setminus S$ is independent. By Corollary \ref{coro: no vertex adajcet to 1 or 2 or 3 or max degree in critical graph}, the neighborhood of $x$ is not one vertex in $S$. According to the cardinality of the neighborhood of $x$ we consider four proof cases.		
	\end{minipage}\hfill
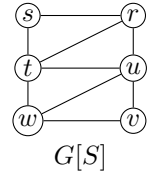
	\begin{minipage}{0.2\textwidth}
		
		\centering
		\begin{tikzpicture}[hhh/.style={draw=black,circle,inner sep=1pt,minimum size=0.2cm},scale=0.7]
			\node[hhh]  (r) at (1,2.5) 	{$r$};
			\node[hhh] 	(s) at (-1,2.5) {$s$};
			\node[hhh] 	(t) at (-1,1.5)	{$t$};
			\node[hhh] 	(u) at (1,1.5) {$u$};
			\node[hhh] 	(v) at (1,0.5) {$v$};
			\node[hhh]  (w) at (-1,0.5) {$w$};
			
			\node 	(g) at (0,-0.2) {$G[S]$};
			\draw (r)--(s)--(t)--(u)--(r)--(t) (u)--(v)--(w) (u)--(w)--(t);
		\end{tikzpicture}
	\end{minipage}
	
	Case $1$: The cardinality of the neighborhood of $x$ is $2$. By Corollary \ref{coro: no vertex adajcet to 1 or 2 or 3 or max degree in critical graph}, $x$ is nonadjacent to two adjacent vertices.
	If $N(w)=\{s,v\}$, then in $G/st$ $\{x\} \cup S \setminus \{r\}$ induces $L_{2}$, which contradicts the assumption that $G$ is a critically non-line graph. 
	Otherwise, $x$ is adjacent to a vertex that is adjacent to two independent vertices while $x$ is nonadjacent to any of them. Hence, the four vertices induces claw while there is two other adjacent vertices outside the claw, which is a contradiction to Theorem \ref{Theorem: critical H-exist: not S is independent}. Thus, $|N(x)| \not= 2$.
	
	Case $2$: The cardinality of the neighborhood of $x$ is $3$. By Corollary \ref{coro: no vertex adajcet to 1 or 2 or 3 or max degree in critical graph}, $x$ is nonadjacent to any three vertices that induce either $P_{3}$ or $C_{3}$. 
	If $N(w)=\{r,s,v\}$ (or similarly $N(w)=\{s,v,w\}$), then in $G/rt$ the set $\{x\} \cup S \setminus \{s\}$ induces $L_{3}$, which contradicts the assumption that $G$ is a critically non-line graph. 
	Otherwise, $x$ is adjacent to a vertex that is adjacent to two independent vertices while $x$ is nonadjacent to any of them. Hence, the four vertices induce claw while there are two other adjacent vertices outside the claw, which is a contradiction to Theorem \ref{Theorem: critical H-exist: not S is independent}. Thus, $|N(x)| \not= 3$.
	
	Case $3$: The cardinality of the neighborhood of $x$ is $4$.
	If $y \in N(x)$, then $(N(x) \setminus \{y\}) \not \subseteq N(y)$, otherwise $x$ is a corner dominated by $y$.
	If $N(x)=\{r,s,v,w\}$, then $G/rs[\{x\} \cup S \setminus \{v\}]$ induces $L_{2}$, which contradicts the assumption that $G$ is a critically non-line. 
	Furthermore, if $N(x)=\{r,s,t,v\}$ (similarly $N(x)=\{s,u,v,w\}$), however, $G/rt[\{x\} \cup S \setminus \{s\}]$ induces $L_{2}$, which contradicts the assumption that $G$ is a critically non-line graph.
	Moreover, if $N(x)=\{r,s,u,v\}$ (similarly $N(x)=\{s,t,v,w\}$), then in $G/tw$ the set $\{,u,v,w,x\}$ induces $L_{2}$, which contradicts the assumption that $G$ is a critically non-line graph.
	Otherwise, $x$ has a neighbor $y$ that is adjacent to two independent vertices while $x$ is nonadjacent to any of them. Hence, the four vertices induces claw while there are two other adjacent vertices outside the claw, which is a contradiction to Theorem \ref{Theorem: critical H-exist: not S is independent}. Thus, $|N(x)| \not= 4$.
	
	Case $4$: The cardinality of the neighborhood of $x$ is at least $5$.
	If $x$ is nonadjacent to $s$ (or $v$), then $x$ is a corner dominated by $t$ (or $u$), which is a contradiction to Theorem \ref{Theorem: critical H-exist: not S is independent}. 
	Furthermore, if $x$ is nonadjacent to $r$ (or similarly $x$ is nonadjacent to $w$), then $\{x\} \cup S \setminus \{r,v\}$ induces $L_{8}$ in $G/ru$, which contradicts the assumption that $G$ is a critically non-line graph. 
	Otherwise, $x$ is nonadjacent to $t$ (or $u$) or adjacent to all vertices in $S$, then $\{x\} \cup S \setminus \{v\}$ induces $L_{8}$ in $G/st$, yielding the same contradiction. Thus, $|N(x)| \not\geq 5$.
	
	Finally, to complete the proof, we prove that $L_{7}$ is a critically non-line. Any $L_{7}$-contraction has independence number $2$, order $5$, and a vertex with degree either $4$ or $1$. Hence any $L_{7}$-contraction is $\{L_{1},L_{2}\}$-free. Thus, $L_{7}$ is a critically non-line graph.
\end{proof}

\subsubsection{The critically non-line graphs that are $L_{8}$-exist}
\begin{lemma}\label{lem: L8 critical line-exist}
	The graph $L_{8}$ is the only critically non-line graph that is $L_{8}$-exist.
\end{lemma}
\begin{proof}
	For the sake of contradiction, we assume that there is a critically non-line graph $G$ that is non-isomorphic to $L_{8}$ where $G$ has a vertex set $S$ that induce $L_{8}$.

	\begin{minipage}{0.75\textwidth}
		Assume that $G[S]$ is labeled as in the figure at the right. By Theorem \ref{Theorem: critical H-exist: not S is independent}, we note that $V(G) \setminus S$ is independent. By Corollary \ref{coro: no vertex adajcet to 1 or 2 or 3 or max degree in critical graph}, we conclude that there is no vertex in $V(G) \setminus S$ whose neighborhood is one vertex in $S$. Furthermore, no vertex in $V(G) \setminus S$ exists that is adjacent to any vertex in $\{r,s,t\}$.
	\end{minipage}\hfill
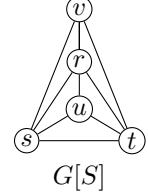
	\begin{minipage}{0.2\textwidth}
		
		\centering
		\begin{tikzpicture}[hhh/.style={draw=black,circle,inner sep=1pt,minimum size=0.2cm},scale=0.7]
			\node[hhh]  (r) at (0,2) 	{$r$};
			\node[hhh] 	(s) at (-1,0.5) {$s$};
			\node[hhh] 	(t) at (1,0.5) 	{$t$};
			\node[hhh] 	(u) at (0,1.1) 	{$u$};
			\node[hhh] 	(v) at (0,3) 	{$v$};
			
			\node 	(g) at (0,-0.2) {$G[S]$};
			
			\draw (r)--(s)--(t)--(r)--(u)--(s) (u)--(t)--(v)--(r)--(v)--(s);
		\end{tikzpicture}
	\end{minipage}
	
	Furthermore, there is no vertex $w$ in $V(G) \setminus S$ that is adjacent to both $u,v$, otherwise, in $G/rt$ the set $\{s,t,u,v,w\}$ induces a graph isomorphic to $L_{2}$, which contradicts that $G$ is a critically non-line graph. Finally to complete our proof, we prove that $L_{8}$ is a critically non-line. Contracting any edge in $L_{8}$ generates a graph with independence number $2$ and order $4$, which is claw-free graph. Thus, $L_{8}$ is a critically non-line.
\end{proof}

\subsubsection{The critically non-line graphs that are $L_{9}$-exist}
\begin{lemma}\label{lem: L9 critical line-exist}
	The graph $L_{9}$ is the only critically non-line graph that is $L_{9}$-exist.
\end{lemma}
\begin{proof}
	For the sake of contradiction, we assume that there is a critically non-line graph $G$ that is non-isomorphic to $L_{9}$ where $G$ has a vertex set $S$ that induces $L_{9}$.

	\begin{minipage}{0.75\textwidth}
		Assume that $G[S]$ is labeled as in the figure at the right. By Theorem \ref{Theorem: critical H-exist: not S is independent}, we note that $V(G) \setminus S$ is independent. By Corollary \ref{coro: no vertex adajcet to 1 or 2 or 3 or max degree in critical graph}, there is no vertex $ V(G) \setminus S$ whose neighborhood is one vertex in $S$, two adjacent vertices, or three vertices inducing $P_{3}$. Furthermore, no vertex in $ V(G) \setminus S$ is adjacent to $r$. By Proposition \ref{pro: claw adn a vertex outside with adajcent vertices inside}, the neighborhood of $x$ is neither a pair of nonadjacent vertices nor three vertices.
	\end{minipage}\hfill
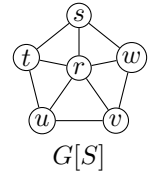
	\begin{minipage}{0.2\textwidth}
		
		\centering
		\begin{tikzpicture}[hhh/.style={draw=black,circle,inner sep=1pt,minimum size=0.2cm},scale=0.7]
			\node[hhh]  (r) at (0,1.5) 	{$r$};
			\node[hhh] 	(s) at (0,2.5) 	{$s$};
			\node[hhh] 	(t) at (-1,1.7)	{$t$};
			\node[hhh] 	(u) at (-0.7,0.5) {$u$};
			\node[hhh] 	(v) at (0.7,0.5) {$v$};
			\node[hhh]  (w) at (1,1.7) 	{$w$};
			
			\node 	(g) at (0,-0.2) {$G[S]$};
			
			\draw (s)--(t)--(u)--(v)--(w)--(s)--(r)--(t) (u)--(r)--(v) (r)--(w);
		\end{tikzpicture}
	\end{minipage}
	
	If the neighborhood of $x$ is four vertices in $S \setminus \{r\}$, say $S \setminus \{r,s\}$, then $\{r,s,t,w,x\}$ induces $L_{2}$ while $u,v$ are adjacent, which is a contradiction to Theorem \ref{Theorem: critical H-exist: not S is independent}. If $x$ is adjacent to every vertex in $S \setminus \{r\}$, then In $G/rs$, $\{r,u,v,w,x\}$ induces $L_{8}$ which contradicts the assumption that $G$ is a critically non-line graph. Finally to complete the proof, we prove that $L_{9}$ is a critically non-line graph. Any $L_{9}$-contraction has independence number $2$ and a maximum degree $4$. Thus, any $L_{9}$-contraction is $\{L_{1},L_{2}\}$-free. Thus, $L_{9}$ is a critically non-line graph.
\end{proof}

%
%
By Lemmas \ref{lem: L1 critical line-exist} to \ref{lem: L9 critical line-exist}, we deduce the following.
\begin{theorem}\label{lem: critical Line-exist graphs}
	The graphs in $\{L_{1},\dots,L_{13}\}-\{L_{3}\}$ are the critically non-line graphs.
\end{theorem}

\subsubsection{Line-Split Graphs}
\begin{theorem}\label{lem: line-split-graphs are line graphs}
	The graphs $L_{14},\dots,L_{34}$ are line graphs.
\end{theorem}
\begin{proof}
	We note that the edge sets of the graphs $L_{14},\dots,L_{34}$ can be partitioned cliques such that every vertex is contained in at most two of such cliques. Thus, by Theorem \ref{lineCharcaterization by harary} the result follows.
\end{proof}

\begin{theorem}\label{lem: unique line-split-graphs}
	The set $\elm(\{L_{14},\dots,L_{34}\})$ is equivalent to the set of graphs $\{L_{14},\dots,\\L_{21}\}$.
\end{theorem}
\begin{proof}
	The graphs $L_{22},\dots,L_{32}$ are $L_{14}$-exist. The graph $L_{33}$ is $L_{20}$-exist. The graph $L_{34}$ is $L_{15}$-exist. Moreover, it is not hard to see that the graphs $\{L_{14},\dots,L_{21}\}$ are minimal.
\end{proof}

\begin{figure}[ht!]
	\centering
	\begin{tikzpicture}[hhh/.style={draw=black,circle,inner sep=1pt,minimum size=0.15cm}]
		
		\begin{scope}[shift={(0,0)}]
			\node 	(graph) at (0:0cm) 	{$L_{14}$};
			\node[hhh]  (a) at (90:1.7cm) 	{};
			\node[hhh] 	(b) at (-0.5,1) 	{};
			\node[hhh] 	(c) at (0.5,1) 	{};
			\node[hhh] 	(d) at (-0.5,0.5) 	{};
			\node[hhh] 	(e) at (0.5,0.5) 	{};
			
			\draw (b)--(a) -- (c) --(b) --  (d)  (e) --(c);
		\end{scope}

		\begin{scope}[shift={(3,0)}]
			\node 	(graph) at (0:0cm) 	{$L_{15}$};
			\node[hhh]  (a) at (90:1.5cm) 	{};
			\node[hhh] 	(b) at (90:2.5cm) 	{};
			\node[hhh] 	(c) at (1,1.5) 	{};
			\node[hhh] 	(d) at (-1,1.5) 	{};
			\node[hhh] 	(e) at (0.5,0.5) 	{};
			\node[hhh] 	(f) at (-0.5,0.5) 	{};
			
			\draw (d)--(a) -- (c) --(b) --  (d) -- (f) -- (e) --(c) (f)--(a)--(e) ;
		\end{scope}
		\begin{scope}[shift={(6,0)}]
			\node 	(graph) at (0:0cm) 	{$L_{16}$};
			\node[hhh]  (a) at (-1,1.5) 	{};
			\node[hhh] 	(b) at (-1,0.5) 	{};
			\node[hhh] 	(c) at (-0.3,1) 	{};
			\node[hhh] 	(d) at (1,1.5) 	{};
			\node[hhh] 	(e) at (1,0.5) 	{};
			\node[hhh] 	(f) at (0.3,1) 	{};
			
			\draw (a) -- (c) --(b) --(a) (d) -- (f) -- (e) --(d) (a)--(d) (b)--(e) (c)--(f) ;
		\end{scope}	
		
		\begin{scope}[shift={(9,0)}]
			\node 	(graph) at (0:0cm) 	{$L_{17}$};
			\node[hhh]  (a) at (-0.5,0.5) 	{};
			\node[hhh] 	(b) at (0.5,0.5) 	{};
			\node[hhh] 	(c) at (0.5,1.5) 	{};
			\node[hhh] 	(d) at (-0.5,1.5) 	{};
			\node[hhh] 	(e) at (-0.5,2.5) 	{};
			\node[hhh] 	(f) at (0.5,2.5) 	{};
			\node[hhh] 	(g) at (0,3) 	{};
			
			\draw (a)--(b)--(c)--(d)--(a)--(c)  (b)--(d)--(e) -- (f) --(c) (f)--(g)--(e) ;
		\end{scope}	
		
		\begin{scope}[shift={(0,-3)}]
			\node 	(graph) at (0:0cm) 	{$L_{18}$};
			\node[hhh]  (a) at (0.5,2) 	{};
			\node[hhh] 	(b) at (-0.5,2) {};
			\node[hhh] 	(c) at (-0.5,1.2)	{};
			\node[hhh] 	(d) at (0.5,1.2) {};
			\node[hhh] 	(e) at (0.5,0.5) {};
			\node[hhh]  (f) at (-0.5,0.5) {};
			\node[hhh]  (g) at (0,1.6) {};

			\draw (a)--(b)--(c)--(d)--(a)--(g)--(c) (d)--(g)--(b) (d)--(e)--(f)--(c)--(e) (d)--(f);
		\end{scope}
		\begin{scope}[shift={(3,-3)}]
			\node 	(graph) at (0:0cm) 	{$L_{19}$};
			\node[hhh]  (a) at (0,0.9) 	{};
			\node[hhh] 	(b) at (-0.5,0.5) {};
			\node[hhh] 	(c) at (0.5,0.5)	{};
			\node[hhh] 	(d) at (-0.5,1.5) {};
			\node[hhh] 	(e) at (0.5,1.5) {};
			\node[hhh]  (f) at (1.1,1.2) {};
			\node[hhh]  (g) at (-1.1,1.2) {};

			\draw (b)--(a)--(d)--(g)--(b)--(d) (a)--(g) (c)--(e)--(f)--(a)--(c)--(f) (f)--(a)--(e)--(d);
		\end{scope}
		\begin{scope}[shift={(6,-3)}]
			\node 	(graph) at (0:0cm) 	{$L_{20}$};
			\node[hhh]  (a) at (-0.5,0.5) 	{};
			\node[hhh] 	(b) at (-0.5,2) 	{};
			\node[hhh] 	(c) at (0.5,2) 	{};
			\node[hhh] 	(d) at (0.5,0.5) 	{};
			\node[hhh] 	(e) at (0,0.8) 	{};
			\node[hhh] 	(f) at (0,1.5) 	{};
			
			\draw (d)--(a) -- (b) --(c) -- (d) (a)--  (f) -- (b)  (c) --(f) --(d)--(a) (d)--(e) --(a) (f)--(e);
		\end{scope}
		\begin{scope}[shift={(9,-3)}]
			\node 	(graph) at (0:0cm) 	{$L_{21}$};
			\node[hhh]  (a) at (-0.5,0.5) 	{};
			\node[hhh] 	(b) at (-0.5,2) 	{};
			\node[hhh] 	(c) at (0.5,2) 	{};
			\node[hhh] 	(d) at (0.5,0.5) 	{};
			\node[hhh] 	(e) at (-1,1.25) 	{};
			\node[hhh] 	(f) at (1,1.25) 	{};
			
			\draw (d)--(a) -- (b) --(c) -- (d) (a)--  (e) -- (b)  (c) --(e) --(d) (d)--(f) --(a) (c)--(f)--(b);
		\end{scope}	
	\end{tikzpicture}	
	\caption{The $\elm$(line-split) graphs}
	\label{fig: elm(Line-split graphs)}
\end{figure}
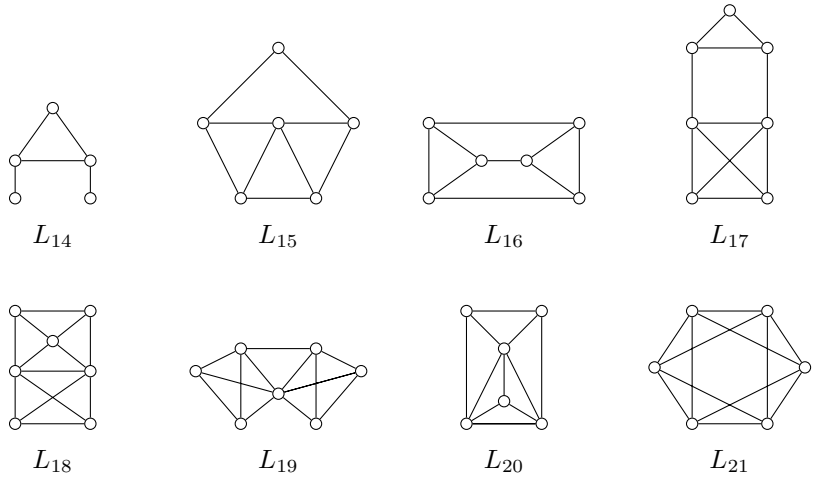
\begin{figure}[ht!]
	\centering
	\begin{tikzpicture}[hhh/.style={draw=black,circle,inner sep=1pt,minimum size=0.15cm}]
		\begin{scope}[shift={(0,0)}]
			\node 	(graph) at (0:0cm) 	{$L_{22}$};
			\node[hhh]  (a) at (90:1.2) 	{};
			\node[hhh] 	(b1) at (0.5,2) 	{};
			\node[hhh] 	(b2) at (-0.5,2) 	{};
			\node[hhh] 	(c) at (1,1.2) 	{};
			\node[hhh] 	(d) at (-1,1.2) 	{};
			\node[hhh] 	(e) at (0.5,0.5) 	{};
			\node[hhh] 	(f) at (-0.5,0.5) 	{};
			
			\draw (d)--(a) -- (c) --(b1) --(b2)--  (d) -- (f) -- (e) --(c) (f)--(a)--(e) ;
		\end{scope}
		\begin{scope}[shift={(3,0)}]
			\node 	(graph) at (0:0cm) 	{$L_{23}$};
			\node[hhh]  (a) at (-1,1.5) 	{};
			\node[hhh] 	(b) at (-1,0.5) 	{};
			\node[hhh] 	(c) at (-0.3,1) 	{};
			\node[hhh] 	(g) at (0,1.5) 	{};
			\node[hhh] 	(d) at (1,1.5) 	{};
			\node[hhh] 	(e) at (1,0.5) 	{};
			\node[hhh] 	(f) at (0.3,1) 	{};
			
			\draw (a) -- (c) --(b) --(a) (d) -- (f) -- (e) --(d) (a)--(g)--(d) (b)--(e) (c)--(f) ;
		\end{scope}	
		
		\begin{scope}[shift={(6,0)}]
			\node 	(graph) at (0:0cm) 	{$L_{24}$};
			\node[hhh]  (a) at (90:1.3cm) 	{};
			\node[hhh] 	(b1) at (0.5,2) 	{};
			\node[hhh] 	(b2) at (-0.5,2) 	{};
			\node[hhh] 	(c) at (1,1.3) 	{};
			\node[hhh] 	(d) at (-1,1.3) 	{};
			\node[hhh] 	(e) at (0.4,0.5) 	{};
			\node[hhh] 	(f) at (-0.4,0.5) 	{};
			
			\draw (d)--(a) -- (c) --(b1) (b2)--  (d) -- (f) -- (e) --(c) (f)--(a)--(e) ;
		\end{scope}
		\begin{scope}[shift={(9,0)}]
			\node 	(graph) at (0:0cm) 	{$L_{25}$};
			\node[hhh]  (a) at (-0.5,1.5) 	{};
			\node[hhh] 	(b) at (-0.5,0.5) 	{};
			\node[hhh] 	(c) at (-1,1) 	{};
			\node[hhh] 	(d) at (0.5,1.5) 	{};
			\node[hhh] 	(e) at (0.5,0.5) 	{};
			\node[hhh] 	(f) at (1,1) 	{};
			\node[hhh] 	(g) at (1.5,1) 	{};
			
			\draw (a) -- (c) --(b) --(a) (d) -- (f) -- (e) --(d) (a)--(d) (b)--(e)   (f)--(g);
		\end{scope}	
		
		\begin{scope}[shift={(0,-4)}]
			\node 	(graph) at (0:0cm) 	{$L_{26}$};
			\node[hhh]  (a) at (-0.5,0.5) 	{};
			\node[hhh] 	(b) at (0.5,0.5) 	{};
			\node[hhh] 	(c) at (0.5,1.5) 	{};
			\node[hhh] 	(d) at (-0.5,1.5) 	{};
			\node[hhh] 	(e) at (-0.5,2.5) 	{};
			\node[hhh] 	(f) at (0.5,2.5) 	{};
			\node[hhh] 	(g) at (0.5,3) 	{};
			
			\draw (a)--(b)--(c)--(d)--(a)--(c)  (b)--(d)--(e) -- (f) --(c)--(e)--(f)--(g) ;
		\end{scope}
		
		\begin{scope}[shift={(3,-4)}]
			\node 	(graph) at (0:0cm) 	{$L_{27}$};
			\node[hhh]  (a) at (-0.5,0.5) 	{};
			\node[hhh] 	(b) at (0.5,0.5) 	{};
			\node[hhh] 	(c) at (0.5,1.5) 	{};
			\node[hhh] 	(d) at (-0.5,1.5) 	{};
			\node[hhh] 	(e) at (0,1) 	{};
			\node[hhh] 	(f) at (0,2.5) 	{};
			\node[hhh] 	(g) at (0,2) 	{};
			
			\draw (a)--(b)--(c)--(d)--(a)  (a)--(e)--(b) (c)--(e)--(d)  (f)--(g)  (c)--(g)--(d) ;
		\end{scope}
		\begin{scope}[shift={(6,-4)}]
			\node 	(graph) at (0:0cm) 	{$L_{28}$};
			\node[hhh]  (a) at (-0.5,0.5) 	{};
			\node[hhh] 	(b) at (0.5,0.5) 	{};
			\node[hhh] 	(c) at (0.5,1.5) 	{};
			\node[hhh] 	(d) at (-0.5,1.5) 	{};
			\node[hhh] 	(e) at (0,1) 	{};
			\node[hhh] 	(f) at (0,2.5) 	{};
			\node[hhh] 	(g) at (0,2) 	{};
			
			\draw (a)--(b)--(c)--(d)--(a)  (a)--(e)--(b) (c)--(e)--(d)  (f)--(g)  (c)--(g)--(d) (g)--(e);
		\end{scope}

		\begin{scope}[shift={(9,-4)}]
			\node 	(graph) at (0:0cm) 	{$L_{29}$};
			\node[hhh]  (a) at (0,1.5) 	{};
			\node[hhh] 	(b1) at (0,2.5) 	{};
			\node[hhh] 	(b2) at (1,2.5) 	{};
			\node[hhh] 	(c) at (1,1.5) 	{};
			\node[hhh] 	(d) at (-1,1.5) 	{};
			\node[hhh] 	(e) at (0.5,0.5) 	{};
			\node[hhh] 	(f) at (-0.5,0.5) 	{};
			
			\draw (d)--(a) -- (c) --(b1)--(b2)--(c)  (b1)--(d)--(f) -- (e) --(c) (f)--(a)--(e) ;
		\end{scope}
		\begin{scope}[shift={(0,-7.5)}]
			\node 	(graph) at (0:0cm) 	{$L_{30}$};
			\node[hhh]  (a) at (0,2) 	{};
			\node[hhh] 	(b1) at (0,2.7) 	{};
			\node[hhh] 	(b2) at (0,0.5) 	{};
			\node[hhh] 	(c) at (1,2) 	{};
			\node[hhh] 	(d) at (-1,2) 	{};
			\node[hhh] 	(e) at (0.5,1) 	{};
			\node[hhh] 	(f) at (-0.5,1) 	{};
			
			\draw (d)--(a) -- (c) --(b1) (e)--(b2)--(f)  (b1)--(d)--(f) -- (e) --(c) (f)--(a)--(e) ;
		\end{scope}	
		\begin{scope}[shift={(3,-7.5)}]
			\node 	(graph) at (0:0cm) 	{$L_{31}$};
			\node[hhh]  (e) at (0,0.5) 	{};
			\node[hhh] 	(d) at (0,1.5) 	{};
			\node[hhh] 	(b) at (-0.5,1.5) 	{};
			\node[hhh] 	(c) at (-0.5,0.5) 	{};
			\node[hhh] 	(a) at (-1,1) {};
			\node[hhh] 	(f) at (0.5,1.5) 	{};
			\node[hhh] 	(g) at (0.5,0.5) {};
			
			\draw (c)--(a)--(b)--(c)--(e)--(d)--(b)  (e)--(g)--(f)--(d) (b)--(e) (c)--(d) (e)--(f);
		\end{scope}
		\begin{scope}[shift={(6,-7.5)}]
			\node 	(graph) at (0:0cm) 	{$L_{32}$};
			\node[hhh]  (a) at (0,1) 	{};
			\node[hhh] 	(b) at (-0.5,1.5) 	{};
			\node[hhh] 	(c) at (-0.5,0.5) 	{};
			\node[hhh] 	(d) at (0.5,1.5) 	{};
			\node[hhh] 	(e) at (0.5,0.5) 	{};
			\node[hhh] 	(f) at (-1,1) 	{};
			\node[hhh] 	(g) at (1,1) 	{};
			
			\draw (b)--(a)--(e)--(c)--(a)--(d)--(b)--(c)  (b)--(f)--(c) (d)--(g)--(e)--(d) ;
		\end{scope}

		\begin{scope}[shift={(9,-7.5)}]
			\node 	(graph) at (0:0cm) 	{$L_{33}$};
			\node[hhh]  (a) at (0,1.5) 	{};
			\node[hhh] 	(b) at (-0.8,1.2) 	{};
			\node[hhh] 	(c) at (-0.5,0.5) 	{};
			\node[hhh] 	(d) at (0.5,0.5) 	{};
			\node[hhh] 	(e) at (0.8,1.2) 	{};
			\node[hhh] 	(f) at (0.5,2) 	{};
			\node[hhh] 	(g) at (-0.5,2) 	{};
			
			\draw (a)--(b)--(c)--(d)--(e)--(a)--(c)--(e)--(b)--(d)--(a)  (b)--(g) --(f) --(e)  
			(g)--(a) --(f);
		\end{scope}
		
		\begin{scope}[shift={(4.5,-10)}]
			\node 	(graph) at (0:0cm) 	{$L_{34}$};
			\node[hhh]  (a) at (0,2) 	{};
			\node[hhh] 	(b) at (-0.3,1.2) 	{};
			\node[hhh] 	(c) at (0.3,1.2) 	{};
			\node[hhh] 	(d) at (-1,1.2) 	{};
			\node[hhh] 	(e) at (1,1.2) 	{};
			\node[hhh] 	(f) at (-0.3,0.5) 	{};
			\node[hhh] 	(g) at (0.3,0.5) 	{};
			
			\draw (b)--(a)--(c) (d)--(a)--(e)  (d)--(b) --(c) --(e) --(a) 
			(d)--(f) --(b) (c)--(f)
			(g) --(b) (c)--(g)--(e)
			(f)--(g);
		\end{scope}
	\end{tikzpicture}	
	\caption{The line-split graphs that are not $\elm$(line-split) graphs}
	\label{fig: not elm Line-split- graphs}
\end{figure}
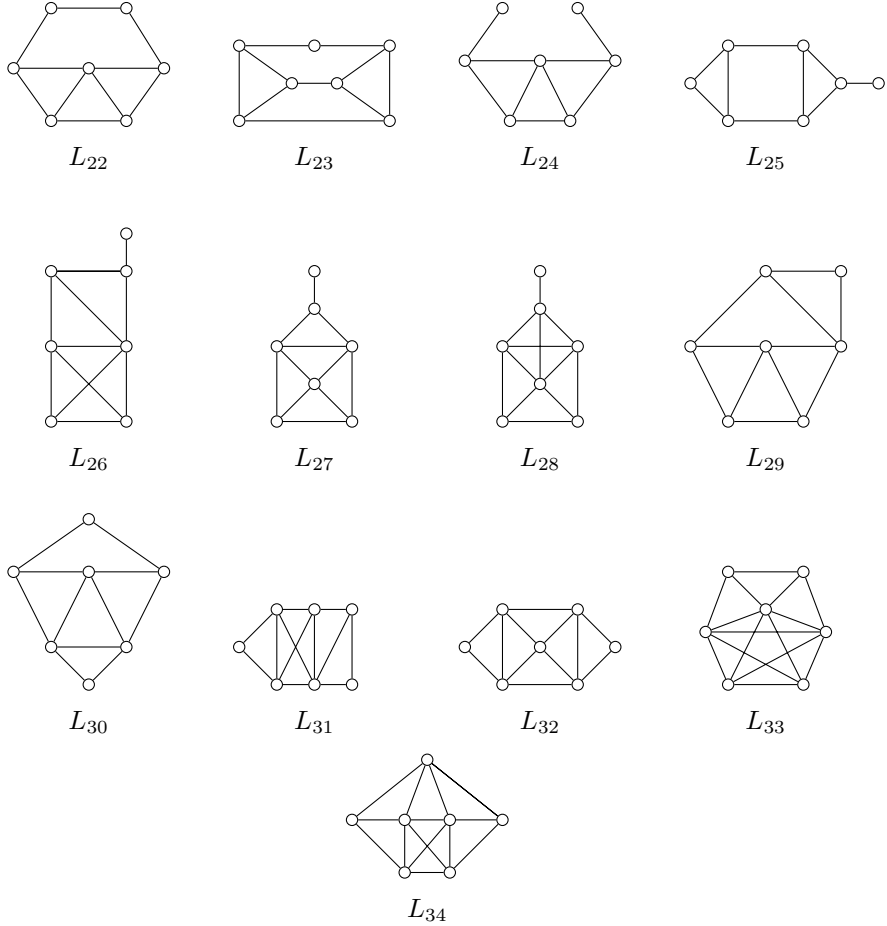

\subsubsection{L1-split graphs that are line}
\begin{lemma}\label{lem: L1-free-split graphs that are line}
	The graph $L_{14}$ is the only $L_{1}$-split graph that is line.
\end{lemma}
\begin{proof}
	To obtain the $L_{1}$-free-split graphs that are line, we construct the set of graphs $splitting(L_{1})$ followed by identifying among it the graphs that are line. 
	
	\begin{minipage}{0.75\textwidth}
		Hence, we partition the vertex set into its orbits generated by the automorphism group. Afterwards, we select an arbitrary vertex from every orbit and construct the set of graphs $splitting(L_{1})$ related to such vertex and identify the line graphs among them.
		The graph $G$ on the right is a labeled $L_{1}$ graph whose vertex set admits the following partition into its orbits $\{\{r\},\{s,t,u\}\}$. W.l.o.g, we choose an arbitrary vertex from every block from the partition. Hence we have two cases.
	\end{minipage}\hfill
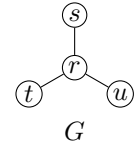
	\begin{minipage}{0.2\textwidth}
		
		\centering
		\begin{tikzpicture}[hhh/.style={draw=black,circle,inner sep=1pt,minimum size=0.2cm},scale=0.7]
			\node[hhh]  (a) at (0:0) 	{$r$};
			\node[hhh] 	(b) at (90:1) 	{$s$};
			\node[hhh] 	(c) at (210:1) 	{$t$};
			\node[hhh] 	(d) at (-30:1) 	{$u$};
			
			\node (graph) at (-90:1.2) {$G$};
			\draw (b)--(a) -- (c) (d) --  (a);
		\end{tikzpicture}
	\end{minipage}
	
	Case 1, vertex $r$: we replace $r$ by two adjacent vertices $r_{1}$ and $r_{2}$. If $r_{1}$ or $r_{2}$ is adjacent to $\{s,t,u\}$ then the constructed graphs is claw-exist. Otherwise, the constructed graphs are either isomorphic to $L_{14}$ or claw-exist. Thus, the line graph in $splitting(G,r)$ is $L_{14}$.
	
	Case 2, vertex $s$: we replace $s$ by two adjacent vertices $s_{1}$ and $s_{2}$. At least one of $s_{1}$ and $s_{2}$ is adjacent to $r$, say $s_{1}$. Hence $\{r,s_{1},t,u\}$ induces claw. Thus, there is no line graph in $splitting(G,s)$. 
\end{proof}

The proofs of Lemma \ref{lem: L2-free-split graphs that are line} to \ref{lem: L9-free-split graphs that are line} are similar to the proof of Lemma \ref{lem: L1-free-split graphs that are line}. Hence, we present sketches of the proofs for the interested readers.
\subsubsection{L2-split graphs that are line}

\begin{lemma}\label{lem: L2-free-split graphs that are line}
	The graphs $L_{15}$ and $L_{16}$ are the only $L_{2}$-split graphs that are line.
\end{lemma}
\begin{proof}
	The graph $G$ on the right is a labeled $L_{2}$ graph whose vertex set admits the following partition into its orbits $\{\{r\},\{s,t\},\{u,v\}\}$.
	
	\begin{minipage}{0.75\textwidth}
		W.l.o.g, we arbitrarily choose a vertex from every block from the partition. Hence we have three cases. 
		
		Case 1,vertex $r$: the $L_{2}$-free-split graph in $splitting(G,r)$ is $L_{3}$. Thus, no $L_{2}$-free-split graph in $splitting(G,r)$ is line. 
	\end{minipage}\hfill
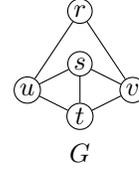
	\begin{minipage}{0.2\textwidth}
		\centering
		\begin{tikzpicture}[hhh/.style={draw=black,circle,inner sep=1pt,minimum size=0.2cm},scale=0.7]
			\node[hhh]  (r) at (0,2.5) 	{$r$};
			\node[hhh] 	(s) at (0,1.5) 	{$s$};
			\node[hhh] 	(t) at (0,0.5)	{$t$};
			\node[hhh] 	(u) at (-1,1) {$u$};
			\node[hhh] 	(v) at (1,1) {$v$};
			
			\node (graph) at (-90:0.2) {$G$};
			
			\draw (v)--(r)--(u) (v)--(s)--(u) (v)--(t)--(u) (t)--(s);
		\end{tikzpicture}
	\end{minipage}
	
	Case 2, vertex $s$: the $L_{2}$-free-split graphs in $splitting(G,s)$ are $L_{1}$-exist or isomorphic to $L_{15}$. Thus, the graph in $splitting(G,s)$ that is line is $L_{15}$.
	
	Case 3, vertex $u$: the $L_{2}$-free-split graphs in $splitting(G,u)$ are $L_{1}$-exist, $L_{3}$-exist, or isomorphic to either $L_{15}$ or $L_{16}$. Thus, the graphs in $splitting(G,u)$ that are line are $L_{15}$ or $L_{16}$. 
\end{proof}

\subsubsection{L3-split graphs that are line}

\begin{lemma}\label{lem: L3-free-split graphs that are line}
	The graphs $L_{22}$ and $L_{23}$ are the only $L_{3}$-split graphs that are line.
\end{lemma}
\begin{proof}
	The graph $G$ on the right is a labeled $L_{3}$ graph whose vertex set admits the following partition into its orbits $\{\{r,w\},\{s,t\},\{u,v\}\}$. W.l.o.g, we choose an arbitrary vertex from every block out of the partition. Hence we have three cases.

	\begin{minipage}{0.75\textwidth}
		Case 1, vertex $r$: the $L_{3}$-free-split graph in $splitting(G,r)$ is a graph that is $L_{4}$-exist. Thus, no $L_{3}$-split graph in $splitting(G,r)$ is line.

		Case 2, vertex $s$: the $L_{3}$-free-split graphs in $splitting(G,s)$ are $L_{1}$-exist or isomorphic to $L_{22}$. Thus, the $L_{3}$-split graph in $splitting(G,s)$ that is line is $L_{22}$.
	\end{minipage}\hfill
	\begin{minipage}{0.2\textwidth}
		
		\centering
		\begin{tikzpicture}[hhh/.style={draw=black,circle,inner sep=1pt,minimum size=0.2cm},scale=0.7]
			\node[hhh]  (a) at (1,1) 	{$r$};
			\node[hhh]  (f) at (-1,1) 	{$w$};
			\node[hhh] 	(b) at (0:0) 	{$s$};
			\node[hhh] 	(c) at (-90:1) 	{$t$};
			\node[hhh] 	(d) at (-1,-0.5) 	{$u$};
			\node[hhh] 	(e) at (1,-0.5) 	{$v$};
			
			\node (graph) at (-90:1.7) {$G$};
			\draw (d)--(b) -- (c) --(d) --  (f) -- (a) -- (e) -- (c) (b) --(e) ;
		\end{tikzpicture}
	\end{minipage}

	Case 3, vertex $u$: the $L_{3}$-free-split graphs in $splitting(G,u)$ are $L_{1}$-exist, $L_{4}$-exist, or isomorphic to either $L_{22}$ or $L_{23}$. Thus, the $L_{3}$-split graphs in $splitting(G,u)$ that are line are $L_{22}$ and $L_{23}$.
\end{proof}

\subsubsection{L4-split graphs that are line}

\begin{lemma}\label{lem: L4-free-split graphs that are line}
	The graphs $L_{24}$ and $L_{25}$ are the only $L_{4}$-split graphs that are line.
\end{lemma}
\begin{proof}
	The graph $G$ is a labeled $L_{4}$ graph whose vertex set admits the following partition into its orbits $\{\{r,w\},\{s,t\},\{u,v\}\}$.
	
	\begin{minipage}{0.75\textwidth}
		W.l.o.g, we choose an arbitrary vertex from every block out of the partition. Hence we have three cases.
		
		Case 1, vertex $r$: there is no $L_{4}$-free-split graph in $splitting(G,r)$. 
		
		Case 2, vertex $s$: the $L_{4}$-free-split graphs in $splitting(G,s)$ are $L_{1}$-exist or isomorphic to $L_{24}$. Thus, the $L_{4}$-split graph in $splitting(G,s)$ that is line is $L_{24}$.
	\end{minipage}\hfill
	\begin{minipage}{0.2\textwidth}
		
		\centering
		\begin{tikzpicture}[hhh/.style={draw=black,circle,inner sep=1pt,minimum size=0.2cm},scale=0.7]
			\node[hhh]  (r) at (-1,2) {$r$};
			\node[hhh]  (w) at (1,2) 	{$w$};
			\node[hhh] 	(s) at (0,1.5) 	{$s$};
			\node[hhh] 	(t) at (0,0.5)	{$t$};
			\node[hhh] 	(u) at (-1,1) {$u$};
			\node[hhh] 	(v) at (1,1) {$v$};
			
			\node (graph) at (-90:0.2) {$G$};
			\draw (v)--(w) (r)--(u) (v)--(s)--(u) (v)--(t)--(u) (t)--(s);
		\end{tikzpicture}
	\end{minipage}
	
	Case 3, vertex $u$: the $L_{4}$-free-split graphs in $splitting(G,u)$ are $L_{1}$-exist or isomorphic to either $L_{24}$ or $L_{25}$. Thus, the $L_{4}$-split graphs in $splitting(G,u)$ that are line are $L_{24}$ and $L_{25}$.
\end{proof}

\subsubsection{L5-split graphs that are line}
\begin{lemma}\label{lem: L5-free-split graphs that are line}
	The graphs $L_{17},L_{26},L_{27}$, and $L_{28}$ are the only $L_{5}$-split graphs that are line.
\end{lemma}
\begin{proof}
	The graph $G$ on the right is a labeled $L_{5}$ graph whose vertex set admits the following partition into its orbits $\{\{r\},\{s\},\{t,u\},\{v,w\}\}$. W.l.o.g, we choose an arbitrary vertex from every block out of the partition. Hence we have four cases.
	
	\begin{minipage}{0.75\textwidth}
		Case 1, vertex $r$: there is no line graph in $splitting(G,r)$.
		
		Case 2, vertex $s$: the $L_{5}$-free-split graphs in $splitting(G,s)$ are $L_{1}$-exist or isomorphic to either $L_{26}$ or $L_{17}$. Thus, $L_{26}$ and $L_{17}$ are the only $L_{5}$-split graphs in $splitting(G,s)$ that are line.
		
		Case 3, vertex $t$: the $L_{5}$-free-split graphs in $splitting(G,t)$ are $L_{1}$-exist, $L_{2}$-exist, or isomorphic to one of the following graphs: $L_{26}$,$L_{27}$, and $L_{28}$. Consequently, the only line graphs in the family of $L_{5}$-split graphs are $L_{26}$, $L_{27}$, and $L_{28}$.
	\end{minipage}\hfill
	\begin{minipage}{0.2\textwidth}
		
		\centering
		\begin{tikzpicture}[hhh/.style={draw=black,circle,inner sep=1pt,minimum size=0.2cm},scale=0.7]
			\node[hhh]  (r) at (0,3.1) 	{$r$};
			\node[hhh] 	(s) at (0,2.3) 	{$s$};
			\node[hhh] 	(t) at (-1,1.5)	{$t$};
			\node[hhh] 	(u) at (1,1.5) {$u$};
			\node[hhh] 	(v) at (1,0.5) {$v$};
			\node[hhh]  (w) at (-1,0.5) 	{$w$};
			
			\node (graph) at (-90:0.2) {$G$};
			\draw (r)--(s)--(t)--(u)--(s) (u)--(v)--(w)--(t)--(v) (u)--(w);
		\end{tikzpicture}
	\end{minipage}
	
	Case 4, vertex $v$: the $L_{5}$-free-split graphs in $splitting(G,v)$ are $L_{1}$-exist or isomorphic to $L_{27}$. Thus, $L_{27}$ is the only $L_{5}$-split graph in $splitting(G,v)$ that is line.
\end{proof}

\subsubsection{L6-split graphs that are line}
\begin{lemma}\label{lem: L6-free-split graphs that are line}
	The graphs $L_{18},L_{19},$ and $L_{33}$ are the only $L_{6}$-split graphs that are line.
\end{lemma}
\begin{proof}
	The graph $G$ on the right is a labeled $L_{6}$ graph whose vertex set admits the following partition into its orbits $\{\{r,s,v,w\},\{t,u\}\}$. W.l.o.g, we choose an arbitrary vertex from every block out of the partition. Hence we have two cases.
	
	\begin{minipage}{0.75\textwidth}
		Case 1, vertex $r$: the $L_{6}$-free-split graphs in $splitting(G,r)$ are $L_{1}$-exist or isomorphic to $L_{18}$. Thus, $L_{18}$ is the only $L_{6}$-split graph in $splitting(G,r)$ that is line.
		
		Case 2, vertex $t$: the $L_{6}$-free-split graphs in $splitting(G,t)$ are $L_{1}$-exist, $L_{2}$-exist, $L_{8}$-exist or isomorphic to either $L_{18}$, $L_{19}$, or $L_{33}$. 
	\end{minipage}\hfill
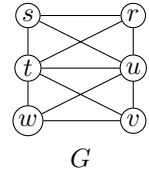
	\begin{minipage}{0.2\textwidth}
		
		\centering
		\begin{tikzpicture}[hhh/.style={draw=black,circle,inner sep=1pt,minimum size=0.2cm},scale=0.7]
			\node[hhh]  (r) at (1,2.5) 	{$r$};
			\node[hhh] 	(s) at (-1,2.5) {$s$};
			\node[hhh] 	(t) at (-1,1.5)	{$t$};
			\node[hhh] 	(u) at (1,1.5) {$u$};
			\node[hhh] 	(v) at (1,0.5) {$v$};
			\node[hhh]  (w) at (-1,0.5) {$w$};
			
			\node (graph) at (-90:0.2) {$G$};
			\draw (r)--(s)--(t)--(u)--(r)--(t) (u)--(s) (u)--(v)--(w)--(t)--(v) (u)--(w);
		\end{tikzpicture}
	\end{minipage}
	Thus, $L_{18}$, $L_{19}$, and $L_{33}$ are the only $L_{6}$-split graphs in $splitting(G,t)$ that are line.
\end{proof}

\subsubsection{L7-split graphs that are line}
\begin{lemma}\label{lem: L7-free-split graphs that are line}
	The graphs $L_{29},L_{30},L_{31}$, and $L_{32}$ are the only $L_{7}$-split graphs that are line.
\end{lemma}
\begin{proof}
	The graph $G$ on the right is a labeled $L_{7}$ graph whose vertex set admits the following partition into its orbits $\{\{r,w\},\{s,v\},\{t,u\}\}$. 
	
	\begin{minipage}{0.75\textwidth}
		W.l.o.g, we choose an arbitrary vertex from every block out of the partition. Hence we have three cases.
		
		Case 1, vertex $r$: the $L_{7}$-free-split graphs in $splitting(G,r)$ are $L_{1}$-exist. Thus, there is no $L_{7}$-split graph in $splitting(G,r)$ that is line.
		
		Case 2, vertex $s$: the $L_{7}$-free-split graph in $splitting(G,s)$ is $L_{1}$-exist. Thus, there is no $L_{7}$-free-split graph in $splitting(G,s)$ that is line.

	\end{minipage}\hfill
	\begin{minipage}{0.2\textwidth}
		
		\centering
		\begin{tikzpicture}[hhh/.style={draw=black,circle,inner sep=1pt,minimum size=0.2cm},scale=0.7]
			\node[hhh]  (r) at (1,2.5) 	{$r$};
			\node[hhh] 	(s) at (-1,2.5) {$s$};
			\node[hhh] 	(t) at (-1,1.5)	{$t$};
			\node[hhh] 	(u) at (1,1.5) {$u$};
			\node[hhh] 	(v) at (1,0.5) {$v$};
			\node[hhh]  (w) at (-1,0.5) {$w$};
			
			\node (graph) at (-90:0.2) {$G$};
			\draw (r)--(s)--(t)--(u)--(r)--(t) (u)--(v)--(w) (u)--(w)--(t);
		\end{tikzpicture}
	\end{minipage}
	
	Case 3, vertex $t$: the $L_{7}$-free-split graphs in $splitting(G,t)$ are $L_{1}$-exist, $L_{2}$-exist, or isomorphic to one of the following graphs: $L_{29},L_{30},L_{31},L_{32}$. Thus, the $L_{7}$-split graphs in $splitting(G,t)$ that are line are $L_{29},L_{30},L_{31},L_{32}$.
\end{proof}

\subsubsection{L8-split graphs that are line}

\begin{lemma}\label{lem: L8-free-split graphs that are line}
	The graphs $L_{20}$ and $L_{21}$ are the only $L_{8}$-split graphs that are line.
\end{lemma}
\begin{proof}
	The graph $G$ on the right is a labeled $L_{8}$ graph whose vertex set admits the following partition into its orbits $\{\{r,s,t\},\{u,v\}\}$. 
	
	\begin{minipage}{0.75\textwidth}
		W.l.o.g, we choose an arbitrary vertex from every block out of the partition. Hence we have two cases.
		
		Case 1, vertex $r$: the $L_{8}$-free-split graphs in $splitting(G,r)$ are $L_{1}$-exist, $L_{2}$-exist, or isomorphic to either $L_{20}$ or $L_{21}$. Thus, the $L_{8}$-split graphs in $splitting(G,r)$ that are line are $L_{20}$ and $L_{21}$.
		
		Case 2, vertex $u$: The $L_{8}$-free-split graphs in $splitting(G,u)$ are $L_{2}$-exist or isomorphic to $L_{20}$. 
	\end{minipage}\hfill
	\begin{minipage}{0.2\textwidth}
		
		\centering
		\begin{tikzpicture}[hhh/.style={draw=black,circle,inner sep=1pt,minimum size=0.2cm},scale=0.7]
			\node[hhh]  (r) at (0,2) 	{$r$};
			\node[hhh] 	(s) at (-1,0.5) {$s$};
			\node[hhh] 	(t) at (1,0.5) 	{$t$};
			\node[hhh] 	(u) at (0,1.1) 	{$u$};
			\node[hhh] 	(v) at (0,3) 	{$v$};
			
			\node (graph) at (-90:0.2) {$G$};
			
			\draw (r)--(s)--(t)--(r)--(u)--(s) (u)--(t)--(v)--(r)--(v)--(s);
		\end{tikzpicture}
	\end{minipage}
	
	Thus, the $L_{8}$-split graph in $splitting(G,u)$ that is line is $L_{20}$.
\end{proof}

\subsubsection{L9-split graphs that are line}
\begin{lemma}\label{lem: L9-free-split graphs that are line}
	The graphs $L_{34}$ is the only $L_{9}$-split graphs that are line.
\end{lemma}
\begin{proof}
	The graph $G$ on the right is a labeled $L_{9}$ graph whose vertex set admits the following partition into its orbits $\{\{r\},\{s,t,u,v,w\}\}$. W.l.o.g, we choose an arbitrary vertex from every block out of the partition. Hence we have two cases.
	
	\begin{minipage}{0.75\textwidth}
		Case 1, vertex $r$: the $L_{9}$-free-split graphs in $splitting(G,r)$ are $L_{1}$-exist or $L_{2}$-exist. Thus, there is no $L_{9}$-split graph in $splitting(G,r)$ that is line.
		
		Case 2, vertex $d$: the $L_{9}$-free-split graphs in $splitting(G,s)$ are $L_{1}$-exist, $L_{2}$-exist, $L_{3}$-exist, or isomorphic to $L_{34}$. 
	\end{minipage}\hfill
	\begin{minipage}{0.2\textwidth}
		
		\centering
		\begin{tikzpicture}[hhh/.style={draw=black,circle,inner sep=1pt,minimum size=0.2cm},scale=0.7]
			\node[hhh]  (r) at (0,1.5) 	{$r$};
			\node[hhh] 	(s) at (0,2.5) 	{$s$};
			\node[hhh] 	(t) at (-1,1.7)	{$t$};
			\node[hhh] 	(u) at (-0.7,0.5) {$u$};
			\node[hhh] 	(v) at (0.7,0.5) {$v$};
			\node[hhh]  (w) at (1,1.7) 	{$w$};
			
			\node (graph) at (-90:0.2) {$G$};
			\draw (s)--(t)--(u)--(v)--(w)--(s)--(r)--(t) (u)--(r)--(v) (r)--(w);
		\end{tikzpicture}
	\end{minipage}
	
	Thus, the $L_{9}$-split graph in $splitting(G,s)$ that is line is $L_{34}$.
\end{proof}

By Theorems \ref{lem: line-split-graphs are line graphs} and \ref{lem: unique line-split-graphs} and Lemmas \ref{lem: L1-free-split graphs that are line} to \ref{lem: L9-free-split graphs that are line}, the following result follows.
\begin{theorem}\label{theorm: line-split graphs}
	The graphs $\{L_{14},\dots,L_{21}\}$ are the $\elm$(line-split) graphs.
\end{theorem}

By Theorems \ref{Theorem: characeterization}, \ref{lem: critical Line-exist graphs}, and \ref{theorm: line-split graphs}, we deduce the following.
\begin{theorem}\label{them: Line graph characterization}
	Let $G$ be a $\{L_{14},\dots,L_{21}\}$-free graph that is non-isomorphic to any graph in $\{L_{1},\dots,L_{13}\}-\{L_{3}\}$. The graph $G$ is line if and only if any $G$-contraction is line.
\end{theorem}

\section{Acknowledgments}
The research presented here is funded by the European Social Fund (ESF).

\bibliographystyle{chicago}
\bibliography{mybib}

\begin{thebibliography}{}

\bibitem[\protect\citeauthoryear{Beineke}{Beineke}{1970}]{beineke1970characterizations}
Beineke, L.~W. (1970).
\newblock Characterizations of derived graphs.
\newblock {\em Journal of Combinatorial Theory\/}~{\em 9\/}(2), 129--135.

\bibitem[\protect\citeauthoryear{Bondy and Murty}{Bondy and
  Murty}{2000}]{bondy2000graph}
Bondy, J.~A. and U.~S. Murty (2000).
\newblock {\em Graph Theory}.
\newblock Springer.

\bibitem[\protect\citeauthoryear{Cameron and Fitzpatrick}{Cameron and
  Fitzpatrick}{2015}]{cameron2015edge}
Cameron, B. and S.~Fitzpatrick (2015).
\newblock Edge contraction and cop-win critical graphs.
\newblock {\em Australasian Journal of Combinatorics\/}~{\em 63}, 70--87.

\bibitem[\protect\citeauthoryear{Chudnovsky and Seymour}{Chudnovsky and
  Seymour}{2008}]{chudnovsky2008claw}
Chudnovsky, M. and P.~Seymour (2008).
\newblock {C}law-free graphs. {IV}. {D}ecomposition theorem.
\newblock {\em Journal of Combinatorial Theory, Series B\/}~{\em 98\/}(5),
  839--938.

\bibitem[\protect\citeauthoryear{Chudnovsky and Seymour}{Chudnovsky and
  Seymour}{2005}]{chudnovsky2005structure}
Chudnovsky, M. and P.~D. Seymour (2005).
\newblock The structure of claw-free graphs.
\newblock {\em Surveys in {C}ombinatorics\/}~{\em 327}, 153--171.

\bibitem[\protect\citeauthoryear{Diner, Paulusma, Picouleau, and Ries}{Diner
  et~al.}{2018}]{diner2018contraction}
Diner, {\"O}.~Y., D.~Paulusma, C.~Picouleau, and B.~Ries (2018).
\newblock Contraction and deletion blockers for perfect graphs and h-free
  graphs.
\newblock {\em Theoretical Computer Science\/}~{\em 746}, 49--72.

\bibitem[\protect\citeauthoryear{Faudree, Flandrin, and
  Ryj{\'a}{\v{c}}ek}{Faudree et~al.}{1997}]{faudree1997claw}
Faudree, R., E.~Flandrin, and Z.~Ryj{\'a}{\v{c}}ek (1997).
\newblock Claw-free graphs a survey.
\newblock {\em Discrete Mathematics\/}~{\em 164\/}(1-3), 87--147.

\bibitem[\protect\citeauthoryear{Foldes and Hammer}{Foldes and
  Hammer}{1976}]{foldes1976class}
Foldes, S. and P.~Hammer (1976).
\newblock Class of matroid-producing graphs.
\newblock Volume~23, pp.\  A420--A420.

\bibitem[\protect\citeauthoryear{Harary}{Harary}{2018}]{harary}
Harary, F. (2018).
\newblock {\em Graph Theory}.
\newblock CRC Press.

\bibitem[\protect\citeauthoryear{Krausz}{Krausz}{1943}]{krausz1943demonstration}
Krausz, J. (1943).
\newblock D{\'e}monstration nouvelle d’une th{\'e}oreme de whitney sur les
  r{\'e}seaux.
\newblock {\em Mat. Fiz. Lapok\/}~{\em 50\/}(1), 75--85.

\bibitem[\protect\citeauthoryear{Kriesell}{Kriesell}{2002}]{kriesell2002survey}
Kriesell, M. (2002).
\newblock A survey on contractible edges in graphs of a prescribed vertex
  connectivity.
\newblock {\em Graphs and Combinatorics\/}~{\em 18\/}(1), 1--30.

\bibitem[\protect\citeauthoryear{Paulusma, Picouleau, and Ries}{Paulusma
  et~al.}{2016}]{paulusma2016reducing}
Paulusma, D., C.~Picouleau, and B.~Ries (2016).
\newblock Reducing the clique and chromatic number via edge contractions and
  vertex deletions.
\newblock In {\em International Symposium on Combinatorial Optimization}, pp.\
  38--49. Springer.

\bibitem[\protect\citeauthoryear{Paulusma, Picouleau, and Ries}{Paulusma
  et~al.}{2019}]{paulusma2019critical}
Paulusma, D., C.~Picouleau, and B.~Ries (2019).
\newblock Critical vertices and edges in {H}-free graphs.
\newblock {\em Discrete Applied Mathematics\/}~{\em 257}, 361--367.

\bibitem[\protect\citeauthoryear{Plummer and Saito}{Plummer and
  Saito}{2014}]{plummer2014note}
Plummer, M.~D. and A.~Saito (2014).
\newblock A note on graphs contraction-critical with respect to independence
  number.
\newblock {\em Discrete Mathematics\/}~{\em 325}, 85--91.

\end{thebibliography}

\end{document}